\documentclass{article}
\usepackage{graphicx} 
\usepackage{amsmath}
\usepackage{amsfonts}
\usepackage{amssymb}
\usepackage[english]{babel}
\usepackage{amsthm}
\usepackage{tikz}
\usetikzlibrary{decorations.pathreplacing,calc}
\usepackage{thm-restate}
\usepackage[margin=1.0in]{geometry}

\newtheorem{proposition}{Proposition}
\newtheorem{definition}{Definition}

\newtheorem{lemma}{Lemma}
\newtheorem{corollary}{Corollary}

\newtheorem{conjecture}{Conjecture}

\title{\textbf{Minimizing Monochromatic Subgraphs of $K_{n,n}$}}
\author{Charles Gong}
\date{}
\begin{document}

\maketitle

\begin{abstract}
Given any $r$-edge coloring of $K_{n,n}$, how large is the maximum (over all $r$ colors) sized monochromatic subgraph guaranteed to be? We give answers to this problem for $r \leq 8$, when $r$ is a perfect square, and when $r$ is one less than a perfect square all up to a constant additive term that depends on $r$. We give a lower bound on this quantity that holds for all $r$ and is sharp when $r$ is a perfect square up to a constant additive term that depends on $r$. Finally, we give a construction for all $r$ which provides an upper bound on this quantity up to a constant additive term that depends on $r$, and which we conjecture is also a lower bound. 
\end{abstract}

\begin{center}
\section{Introduction}
\end{center}

Let $K_{n,n}$ denote the complete bipartite graph with $n$ vertices in each of the partite sets, and let $A_n = \{u_i : i \in [n]\} \subseteq V(K_{n,n})$ denote the vertex set of one partite set of $K_{n,n}$, and $B_n = \{v_i : i \in [n]\} \subseteq V(K_{n,n})$ denote the vertex set of the other partite set of $K_{n,n}$. An $r$-edge coloring of the complete bipartite graph $K_{n,n}$ is a function  $f : E(K_{n,n}) \rightarrow [r]$, where $E(K_{n,n})$ is the set of edges of $K_{n,n}$. Let $V(K_{n,n})$ denote the vertex set of $K_{n,n}$, and $[r]^{E(K_{n,n})}$ denote the set of functions from $E(K_{n,n})$ to $[r]$.  Given a coloring $f :  E(K_{n,n}) \rightarrow [r]$, we'll say that a color $i \in [r]$ \textbf{touches} a vertex $v \in V(K_{n,n})$ and a vertex $v$ \textbf{touches} a color $i$ if there exists an edge $e \in E(K_{n,n})$ such that $v \in e$ and $f(e) = i$; i.e. $v$ is incident to an edge $e$ of color $i$. Although touching is defined with respect to a specific coloring $f$, when we use the term ``touch" we will omit what coloring $f$ it refers to, as it will be clear from context. Let $\mathbb{Z}^{+}$ denote the set of positive integers. \\

\begin{definition}
Let
$$g(n,r) := \min_{f \in [r]^{E(K_{n,n})}} \ \max_{i \in [r]} |\{v \in V(K_{n,n}) : i \text{ touches } v \}|$$

\noindent
for $n,r \in \mathbb{Z}^{+}$.
\end{definition}

\noindent
So the goal is to minimize the number of vertices any color touches in an $r$-edge coloring of $K_{n,n}$. \\ 

Alon, Bucić, Christoph, and Krivelevich \cite{alon} studied an analogous problem for complete graphs $K_n$, as opposed to complete bipartite graphs $K_{n,n}$. In particular, they studied the quantity $g(n,r,s) : (\mathbb{Z}^{+})^3 \rightarrow \mathbb{Z}^{+}$ where

$$g(n,r,s) := \min_{f \in [r]^{E(K_{n})}} \ \max_{S \subseteq [r], \ |S| = s} |\{v \in V(K_{n}) : \exists i \in S \text{ such that $i$ touches $v$}\}|.$$

In other words, instead of considering subgraphs of only one color, they allowed subgraphs to contain $s$ colors. They did this in order to provide upper bounds on the quantity $f(n,r,s) : (\mathbb{Z}^{+})^3 \rightarrow \mathbb{Z}^{+}$, where

 $$ f(n,r,s) = \min_{f \in [r]^{E(K_n)}} \max_{S \subseteq [r], \ |S| = s} \text{size of the largest connected component using colors in $S$},$$ 

\noindent
 which was studied by Liu, Morris, and Prince \cite{liu}. Note that the analogous problem for $g(n,r)$ which deals with monochromatic \textit{connected} subgraphs in $K_{n,m}$ has already been solved by Liu, Morris, and Prince \cite{liu}: in an $r$-edge coloring of $K_{n,m}$, there exists a monochromatic connected component of size at least $\frac{m+n}{r}$, and this lower bound is sharp when $r \mid m$ and $r \mid n$. Somewhat surprisingly, the answer is less simple when one drops the condition that the monochromatic subgraph needs to be connected, and we will be studying this problem with the dropped condition in this paper.  \\

We now describe the results of this paper. In Section 3 we describe an algorithm that for any $n,r \in \mathbb{Z}^{+}$ can compute $g(n,r)$ up to a constant factor depending only on $r$. The algorithm takes in a positive integer $r \in \mathbb{Z}^{+}$ and outputs a rational number in $[0,2]$, and we denote it by $g^* : \mathbb{Z}^{+} \rightarrow [0,2]$. \\

\begin{restatable}{theorem}{computelower}
We have $g(n,r) \geq g^*(r) \cdot n$ for all $n,r \in \mathbb{Z}^{+}$.
\end{restatable}

\begin{restatable}{theorem}{compute}
For all $r \in \mathbb{Z}^{+}$, there exists $N \in \mathbb{Z}^{+}$ such that for all $t \in \mathbb{Z}^{+}$ we have $g(tN,r) = g^*(r) \cdot tN$.
\end{restatable}

\begin{restatable}{theorem}{computeupper}
For all $n,r \in \mathbb{Z}^{+}$, we have $g(n,r) \leq g^*(r) \cdot n + C_r$ for some constant $C_r$ depending on $r$. 
\end{restatable}

Theorems 1, 2 and 3 say we can basically just study $g^*$ instead of $g$, which we will do in further sections. In Section 4 we show

\begin{restatable}{theorem}{squarelower}
We have $g^*(r) \geq \frac{2}{\sqrt{r}}$ for all $r \in \mathbb{Z}^{+}$. 
\end{restatable}

\noindent
In Section 4 we also show 

\begin{restatable}{theorem}{square}
We have $g^*(t^2) = \frac{2}{t}$ for all $t \in \mathbb{Z}^{+}$. 
\end{restatable}

\noindent
 In Section 5 we determine the values of $g^*(r)$ for $r \leq 8$ up to a constant additive term depending on $r$. In Section 6 we show 

\begin{restatable}{theorem}{minus}
We have $g^*(t^2-1) = \frac{1}{t} - \frac{1}{t^2} + \frac{1}{t-1}$ for all $t \geq 3$ where $t \in \mathbb{Z}^{+}$.
\end{restatable}

\noindent
In Section 7 we show 

\begin{restatable}{theorem}{univone}
Let $r \in \mathbb{Z}^{+}$. Suppose there exists $t \in \mathbb{Z}^{+}$ such that $t^2 \leq r \leq t(t+1)$. Then

$$g^*(r) \leq  \frac{2}{t} + \frac{1}{t+1} + \frac{r}{t(t+1)} - \frac{r}{t^2}.$$
\end{restatable}

\begin{restatable}{theorem}{univtwo}
Let $r \in \mathbb{Z}^{+}$. Suppose there exists $t \in \mathbb{Z}^{+}$ such that $t(t+1) \leq r \leq (t+1)^2$. Then

$$g^*(r) \leq  \frac{2}{t+1}+\frac{1}{t}+\frac{r}{(t+1)^2}-\frac{r}{t(t+1)}.$$
\end{restatable}

\noindent
Note Theorems 7 and 8 cover all $r \in \mathbb{Z}^{+}$ (they are exhaustive). In the final Section we will discuss conjectures and a technique not used in the main results of the paper. 

\begin{center}
\section{Coloring Squares}
\end{center}

We first provide a visual representation of an edge coloring $f : E(K_{n,n}) \rightarrow [r]$. Consider when $r = 2$ and $n = 4$. Recall $A_4 = \{u_1,u_2,u_3,u_4\}$ and $B_4 = \{v_1,v_2,v_3,v_4\}$. Consider the coloring $f : E(K_{4,4}) \rightarrow [2]$ defined for $j,k \in \{1,2,3,4\}$ as $f(\{u_j,v_k\}) = 1$ if $j \leq 2$ and $f(\{u_j,v_k\}) = 2$ if $j > 2$. The coloring $f$ can be represented visually by the figure below: \\

\begin{center}
\begin{tikzpicture}
\draw[step=1cm,black] (0,0) grid (4,4);
\draw (0.5,0.5) node{\Large 1};
\draw (0.5,1.5) node{\Large 1};
\draw (0.5,2.5) node{\Large 1};
\draw (0.5,3.5) node{\Large 1};
\draw (1.5,0.5) node{\Large 1};
\draw (1.5,1.5) node{\Large 1};
\draw (1.5,2.5) node{\Large 1};
\draw (1.5,3.5) node{\Large 1};
\draw (2.5,0.5) node{\Large 2};
\draw (2.5,1.5) node{\Large 2};
\draw (2.5,2.5) node{\Large 2};
\draw (2.5,3.5) node{\Large 2};
\draw (3.5,0.5) node{\Large 2};
\draw (3.5,1.5) node{\Large 2};
\draw (3.5,2.5) node{\Large 2};
\draw (3.5,3.5) node{\Large 2};
\draw (0.5, -0.5) node{\Large $u_1$};
\draw (1.5, -0.5) node{\Large $u_2$};
\draw (2.5, -0.5) node{\Large $u_3$};
\draw (3.5, -0.5) node{\Large $u_4$};
\draw (-0.5, 0.5) node{\Large $v_1$};
\draw (-0.5, 1.5) node{\Large $v_2$};
\draw (-0.5, 2.5) node{\Large $v_3$};
\draw (-0.5, 3.5) node{\Large $v_4$};

\draw (2, -1.5) node{\normalsize Figure 1};
\end{tikzpicture}
\end{center} 

\noindent
Since $f(\{u_2,v_3\}) = 1$, the cell at which the column labeled $u_2$ intersects with the row labeled $v_3$ is filled in with the color $1$. Call this the \textbf{coloring square} of $f$. In general, given $f : E(K_{n,n}) \rightarrow [r]$, the coloring square of $f$ is the map $F : A_n \times B_n \rightarrow [r]$, where $F(u,v) = f(\{u,v\})$, but it helps to look at it visually as a matrix. Also, define the projections $\pi_a : A_n \times B_n \rightarrow A_n$ and $\pi_b : A_n \times B_n \rightarrow B_n$ by $\pi_a(u,v) = u$ and $\pi_b(u,v) = v$. \\

 For $u \in A_n$ the \textbf{column} $u$ is the set $\pi_a^{-1}(u) = \{u\} \times B_n$ and for $v \in B_n$ the \textbf{row} $v$ is the set $\pi_b^{-1}(v) = A_n \times \{v\}$. Given a coloring $f : E(K_{n,n}) \rightarrow [r]$ and its coloring square $F: A_n \times B_n \rightarrow [r]$, a column $\{u\} \times B _n$ \textbf{contains} a color $i \in [r]$ if $i \in F(\{u\} \times B_n)$, where $F(\{u\} \times B_n)$ is the image of $F$ on the set $\{u\} \times B_n$, and likewise a row $A_n \times \{v\}$ \textbf{contains} a color $i$ if $i \in F(A_n \times \{v\})$, where $F(A_n \times \{v\})$ is the image of $F$ on the set $A_n \times \{v\}$. To read from $f$'s coloring square $F$ how many vertices color $i \in [r]$ touches, one counts the number of columns and rows containing $i$, i.e.

$$|\pi_a(F^{-1}(i))| + |\pi_b(F^{-1}(i))|.$$

\noindent
 For example, in Figure $1$ all four rows contain color $1$ and columns $u_1,u_2$ contain color $1$, so we have $6$ vertices in total that touch color $1$. Note in Figure 1 we also have that color $2$ touches $6$ vertices, as it is also contained in $6$ rows and columns, and so Figure 1 shows that $g(4,2) \leq 6$, as it is a construction where all colors touch less than or equal to $6$ vertices. \\

Note that the construction above can be generalized for larger (and smaller) $n$. See Figure 2 below. \\

\begin{center}
\begin{tikzpicture}[declare function={a=4;}]
\draw (0,0) coordinate (A) ;
\draw (0,a) coordinate (B) ;
\draw (a,0) coordinate (C) ;
\draw (a,a) coordinate (D) ;
\draw (a/2,0) coordinate (E) ;
\draw (a/2,a) coordinate (F) ;

\draw (E) -- (F) ;
\draw (A) -- (B) ;
\draw (A) -- (C) ;
\draw (B) -- (D) ;
\draw (C) -- (D) ;

\draw (a/4, a/2) node{\huge 1} ;
\draw (3*a/4, a/2) node{\huge 2} ;

\draw[decorate, decoration={brace,raise=3pt,amplitude=4pt}] (A) -- (B) node[midway,xshift=-0.6cm]{\Large $n$};

\draw[decorate, decoration={brace,mirror,raise=3pt,amplitude=4pt}](A) -- (E) node[midway,yshift=-0.6cm]{\Large $\frac{n}{2}$};

\draw[decorate, decoration={brace,mirror,raise=3pt,amplitude=4pt}] (E) -- (C) node[midway,yshift=-0.6cm]{\Large $\frac{n}{2}$};

\draw (a/2, -1.5) node{\normalsize Figure 2} ;
\end{tikzpicture}
\end{center} 

\noindent
Figure 2 is to be interpreted as follows. All cells lying in the region labeled $1$ have color $1$, and all cells lying in the region labeled $2$ have color $2$. We have omitted the vertex labels of the columns and rows, and we have omitted the delineations between distinct cells. There are $n$ rows containing color $1$ and $\frac{n}{2}$ columns containing color $1$, so in total color $1$ touches $\frac{3n}{2}$ vertices. Likewise, there are $\frac{3n}{2}$ rows and columns containing color $2$, so color $2$ also touches $\frac{3n}{2}$ vertices. The construction above shows that $g(n,2) \leq \frac{3n}{2}$ when $2 \mid n$. From now on, we will drop the $n$ in our diagrams and consider \textit{proportions} of rows and columns---see Figure 3 below. \\

\begin{center}
\begin{tikzpicture}[declare function={a=4;}]
\draw (0,0) coordinate (A) ;
\draw (a,0) coordinate (B) ;
\draw (0,a) coordinate (C) ;
\draw (a,a) coordinate (D) ;
\draw (a/2,0) coordinate (E) ;
\draw (a/2,a) coordinate (F) ;

\draw (A) -- (B) ;
\draw (A) -- (C) ;
\draw (C) -- (D) ;
\draw (B) -- (D) ;
\draw (E) -- (F) ;

\draw (a/4, a/2) node{\huge 1} ;
\draw (3*a/4, a/2) node{\huge 2} ;

\draw[decorate, decoration={brace,raise=3pt,amplitude=4pt}] (A) -- (C) node[midway,xshift=-0.6cm]{\Large 1};
\draw[decorate, decoration={brace,mirror,raise=3pt,amplitude=4pt}] (A) -- (E) node[midway,yshift=-0.6cm]{\Large $\frac{1}{2}$};
\draw[decorate, decoration={brace,mirror,raise=3pt,amplitude=4pt}] (E) -- (B) node[midway,yshift=-0.6cm]{\Large $\frac{1}{2}$};

\draw (a/2,-1.5) node{\normalsize Figure 3} ;
\end{tikzpicture}
\end{center}

\begin{center}
\section{Computing $g(n,r)$}
\end{center}

\begin{definition}
Given a coloring $f : E(K_{n,n}) \rightarrow [r]$, for $i \in [r]$ define 

$$a_i = \frac{1}{n} |\{u \in A_n : u \text{ touches } i\}|.$$
\end{definition}

\noindent
In other words, $a_i$ is the proportion of vertices in $A_n$ touching color $i$ and also the proportion of columns containing color $i$. 

\begin{definition}
Let $$b_i = \frac{1}{n} |\{v \in B_n : v \text{ touches } i\}|,$$
\end{definition}

\noindent
In other words, $b_i$ is the proportion of vertices in $B_n$ touching color $i$ and also the proportion of rows containing color $i$. Note the total number of vertices that color $i \in [r]$ touches is $(a_i+b_i)n$. \\

Let functions $a,b : \mathcal{P}([r]) \rightarrow [0,1]$, where $\mathcal{P}([r])$ is the power-set of $[r]$ and $[0,1] \subseteq \mathbb{R}$ is the closed interval from $0$ to $1$. \\

\begin{definition}
For $R \subseteq [r]$, define $$a(R) = \frac{1}{n} |\{u \in A_n : F(\{u\} \times B _n) = R\}|.$$
\end{definition}

\noindent
In other words, $a(R)$ is the proportion of vertices in $A_n$ that touches exactly the colors in $R$ and also the proportion of columns whose set of colors equals $R$. \\

\begin{definition}
For $R \subseteq [r]$, define $$b(R) = \frac{1}{n} |\{v \in B_n : F(A_n \times \{v\}) = R\}|.$$
\end{definition}

\noindent
In other words, $b(R)$ is the proportion of vertices in $B_n$ which touches exactly the colors in $R$ and also the proportion of rows whose set of colors equals $R$. \\

For example, in Figure 3 we have $a(\{1\}) = a(\{2\}) = \frac{1}{2}$ because $\frac{1}{2}$ of the columns contain \textit{only} the color $1$, and $\frac{1}{2}$ of the columns contain \textit{only} the color $2$, but $a(\{1,2\}) = 0$ since no column contains \textit{both} the colors $1$ and $2$. We have $b(\{1,2\}) = 1$ because all rows contain \textit{both} the colors $1$ and $2$, but $b(\{1\}) = b(\{2\}) = 0$ because no row contains \textit{only} the color $1$ or \textit{only} the color $2$. Note that we always have $a(\{\}) = b(\{\}) = 0$ because every row and column must contain at least one color if we are to color all the edges. Additionally, we have that for $R_1,R_2 \subseteq [r]$ with $R_1 \cap R_2 = \{\}$, we must have $a(R_1) = 0$ or $b(R_2) = 0$, otherwise there'd be a column $\{u\} \times B _n$ with colors $R_1 = F(\{u\} \times B _n)$ and a row $A_n \times \{v\}$ with colors $R_2 = F(A_n \times \{v\})$, so the cell $(\{u\} \times B _n) \cap (A_n \times \{v\}) = (u,v)$ at which they intersect cannot be colored as $R_1 \cap R_2 = \{\}$, i.e. $F$ would be undefined at $(u,v)$. \\

We will identify some constraints on $a_i,b_i,a(R),b(R)$ that will allow us to compute $g(n,r)$ up to a constant depending on $r$ with multiple iterations of a linear program. \\

$$\sum_{R \subseteq [r]} a(R) = \sum_{R \subseteq [r]} b(R) = 1$$

\noindent
which comes from 

$$\sum_{R \subseteq [r]} a(R) \cdot n = \sum_{R \subseteq [r]} |\{u \in A_n : F(\{u\} \times B _n) = R\}| = |A_n| = n$$

\noindent
and dividing both sides by $n$. A similar argument shows $\sum_{R \subseteq [r]} b(R) = 1$. We also have for $i \in [r]$ that

$$a_i = \sum_{R \subseteq [r], i \in R} a(R)$$
$$b_i = \sum_{R \subseteq [r], i \in R} b(R).$$

Under these conditions, we almost have a linear program with variables $a_i^*,b_i^*$ for $i \in [r]$ and $a(R)^*,b(R)^*$ for $R \subseteq [r]$. The $^*$ is to distinguish variables in a linear program from values $a_i,b_i,a(R),b(R)$ corresponding to a coloring $f$. In particular, for $R \subseteq [r]$ we have the linear constraints

\begin{equation}
\begin{split}
0 &\leq a(R)^* \leq 1 \\
0 &\leq b(R)^* \leq 1.
\end{split}
\end{equation}

\noindent
We have the linear constraints

\begin{equation} 
\sum_{R \subseteq [r]} a(R)^* = \sum_{R \subseteq [r]} b(R)^* = 1.
\end{equation}

\noindent
For $i \in [r]$, we have the linear constraints

\begin{equation}
\begin{split}
a_i^* &= \sum_{R \subseteq [r], i \in R} a(R)^* \\
b_i^* &= \sum_{R \subseteq [r], i \in R} b(R)^*.
\end{split}
\end{equation}

\noindent
We define a new variable $m$ where for all $i \in [r]$, we have 

\begin{equation}
m \geq a_i^* + b_i^*.
\end{equation}

\noindent
The only non-linear constraint we have is:

\begin{equation}
\text{If $R_1,R_2 \subseteq [r]$ and $R_1 \cap R_2 = \{\}$, then $a(R_1)^* = 0$ or $b(R_2)^* = 0$}
\end{equation}

\noindent
and the linear objective is 

\begin{equation}
\text{Minimize $m$.}
\end{equation}

We now show how we can resolve the non-linear constraint (5). Define the function $h_r : \mathcal{P}(\mathcal{P}([r])) \times \mathcal{P}(\mathcal{P}([r])) \rightarrow [0,2]$, where for $(P_1,P_2) \in  \mathcal{P}(\mathcal{P}([r])) \times \mathcal{P}(\mathcal{P}([r]))$, we have that $h_r(P_1,P_2) = 2$ if there exists $R_1 \in P_1$ and $R_2 \in P_2$ such that $R_1 \cap R_2 = \{\}$, otherwise $h_r(P_1,P_2)$ is the value of $m$ in the linear program with linear constraints (1), (2), (3), (4) and linear objective (6) along with additional linear constraints $a(R)^* = 0$ for all $R_1 \not \in P_1$ and $b(R)^* = 0$ for all $R_2 \not \in P_2$, which is feasible because the constraints give rise to a compact space and the objective is continuous with respect to all the variables (and so $h_r$ is well-defined). We define a function $g^* : \mathbb{Z}^{+} \rightarrow [0,2]$.

\begin{definition}
Let $$g^*(r) = \min_{(P_1,P_2) \in  \mathcal{P}(\mathcal{P}([r])) \times \mathcal{P}(\mathcal{P}([r]))} h_r(P_1,P_2).$$  
\end{definition}

We claim $g^*(r)$ is precisely the minimum value $m'$ of $m$ in the original problem subject to the linear constraints (1), (2), (3), (4), and the non-linear constraint (5), which we will call the \textbf{coloring constraints on $r$ colors}, but we will drop the ``on $r$ colors" part as it will be clear from context. We call the coloring constraints along with linear objective (6) the \textbf{coloring problem on $r$ colors}, where we will also drop the ``on $r$ colors" part as it will be clear from context.\\

We have $g^*(r) \geq m'$ because for all $(P_1,P_2) \in  \mathcal{P}(\mathcal{P}([r])) \times \mathcal{P}(\mathcal{P}([r]))$ we have $h_r(P_1,P_2) \geq m'$, since either there exists $R_1 \in P_1$ and $R_2 \in P_2$ such that $R_1 \cap R_2 = \{\}$ in which case $h_r(P_1,P_2) = 2 \geq m'$ (note $2 \geq m'$ because we can set $a(\{1\})^* = b(\{1\})^* = 1$ and $a(R)^* = b(R)^* = 0$ for all $R \in \mathcal{P}(R) \setminus \{\{1\}\}$, and $a_1^* = b_1^* = 1$ and $a_i^* = b_i^* = 0$ for all $i \in [r] \setminus \{1\}$, and finally we set $m = 2$). Otherwise, there exists a solution $a(R)^*, b(R)^*, a_i^*, b_i^*, m$ satisfying the coloring constraints with $h_r(P_1,P_2) = m \geq m'$. \\

We now show $g^*(r) \leq m'$. Note there exists a solution $a(R)^*,b(R)^*,a_i^*,b_i^*,m'$ satisfying the coloring constraints. Define $P_1 = \{R \subseteq [r] : a(R)^* \neq 0\}$ and $P_2 = \{R \subseteq [r] : b(R)^* \neq 0 \}$. This gives us $h_r(P_1,P_2) \leq m'$ by constraint (5) and the definition of $h_r$.  \\

One can see that $g^*(r)$ is computable by iterating over all $2^{2^r} \times 2^{2^r}$ pairs $(P_1,P_2) \in  \mathcal{P}(\mathcal{P}([r])) \times \mathcal{P}(\mathcal{P}([r]))$ and taking the minimum $h_r(P_1,P_2)$; we know $h_r(P_1,P_2)$ is computable because checking if there exists $R_1 \in P_1$ and $R_2 \in P_2$ such that $R_1 \cap R_2 = \{\}$ is computable, and if there doesn't exist such $R_1,R_2$, then $h_r(P_1,P_2)$ is computed by a linear program. \\

\computelower*

\begin{proof}
Consider any coloring $f : E(K_{n,n}) \rightarrow [r]$. It has a coloring square with corresponding values $a_i,b_i$ for $i \in [r]$ and $a(R),b(R)$ for $R \subseteq [r]$ which satisfies constraints (1), (2), (3), (5), where we set $a_i^* = a_i, b_i^* = b_i, a(R)^* = a(R)$ and $b(R)^* = b(R)$. Furthermore, if we set $m = \max_{i \in [r]} a_i + b_i$, then we also satisfy constraint (4). Thus, 

$$\max_{i \in [r]} |\{v \in V(K_{n,n}) : i \text{ touches } v\}| = \max_{i \in [r]} (a_i + b_i)n = m \cdot n \geq g^*(r) \cdot n.$$

\noindent
Note that this holds for all colorings $f : E(K_{n,n}) \rightarrow [r]$, so in particular

$$\min_{f \in [r]^{E(K_{n,n})}} \ \max_{i \in [r]} |\{v \in V(K_{n,n}) : i \text{ touches } v \}| \geq g^*(r) \cdot n$$

\noindent
and therefore $g(n,r) \geq g^*(r) \cdot n$, as required. 
\end{proof}

So we know $g^*(r)$ provides a lower bound to our original question. But is it tight? The answer is yes under certain divisibility conditions. \\

\compute*

\begin{proof}
Note there exists rational solutions to $m = g^*(r)$, $a_i^*,b_i^*$ for $i \in [r]$, and $a(R)^*,b(R)^*$ for $R \subseteq [r]$ to the coloring problem. This is because the coefficients of the variables in the linear constraints in (1), (2), (3), (4) and linear objective in (6) are all rational. It is a fact that a feasible linear program with rational coefficients has a rational optimal solution \cite{ration}. Thus for all $(P_1,P_2) \in \mathcal{P}(\mathcal{P}([r])) \times \mathcal{P}(\mathcal{P}([r]))$, there exists a solution $m=h(P_1,P_2),a(R)^*,b(R)^*,a_i^*,b_i^*$ satisfying the constraints (1), (2), (3) and (4), where all the variables are rational. Recall $g^*(r)$ is the minimum of $h_r(P_1,P_2)$ over all $(P_1,P_2) \in \mathcal{P}(\mathcal{P}([r])) \times \mathcal{P}(\mathcal{P}([r]))$, and thus there exists $(P_1,P_2) \in \mathcal{P}(\mathcal{P}([r])) \times \mathcal{P}(\mathcal{P}([r]))$ such that $m=g^*(r)=h(P_1,P_2),a(R)^*,b(R)^*,a_i^*,b_i^*$ is a solution to the coloring problem where all the variables are rational. \\

So let $m$, $a_i^*,b_i^*$ for $i \in [r]$, and $a(R)^*,b(R)^*$ for $R \subseteq [r]$ be a rational solution to the coloring problem. Let $N$ be the least common multiple of all the denominators of $a(R)^*$ and $b(R)^*$ for $R \subseteq [r]$. For $t \in \mathbb{Z}^{+}$, we can construct a coloring $f : E(K_{tN,tN}) \rightarrow [r]$ such that 

$$\max_{i \in [r]} |\{v \in V(K_{tN,tN}) : i \text{ touches } v\}| \leq m \cdot n.$$

\noindent
Because $N$ was chosen to be the least common multiple of the denominators of $a(R)^*$ and $b(R)^*$ for $R \subseteq [r]$ and because of constraints (1) and (2), we can partition the $tN$ vertices in $A_{tN}$ into $A_{tN} = \bigsqcup_{R \subseteq [r]} A_R$ (here $\sqcup$ means disjoint union) where for $R \subseteq [r]$ we have $|A_R| = a(R)^* \cdot tN$, and likewise we can partition the $tN$ vertices in $B_{tN}$ into $B_{tN} = \bigsqcup_{R \subseteq [r]} B_R$ where for $R \subseteq [r]$ we have $|B_R| = b(R)^* \cdot tN$. For vertex $u \in A_{R_1}$ and vertex $v \in B_{R_2}$ where $R_1,R_2 \subseteq [r]$, we define $f(u,v) = i$ where $i$ can be anything in $R_1 \cap R_2$. Note $R_1 \cap R_2 \neq \{\}$ because of constraint (5). \\

We claim that $a_i \leq a_i^*$. This inequality is easier to see in the form $a_i \cdot tN \leq a_i^* \cdot tN$. First note that if for some $u \in A_{tN}$ we had $i \in F(\{u\} \times B_{tN})$, i.e. $u$ touches $i$, this must mean $u \in A_R$ for some $R \subseteq [r], i \in R$ by the definition of $f$. Thus we have 

$$\{u \in A_{tN} : u \text{ touches } i \} \subseteq \bigsqcup_{R \subseteq [r], \ i \in R} A_R$$

\noindent
and so 

$$a_i \cdot n = |\{u \in A_{tN} : u \text{ touches } i \}| \leq |\bigsqcup_{R \subseteq [r], \ i \in R} A_R|  = \sum_{R \subseteq [r], \ i \in R} a(R)^* \cdot tN= a_i^* \cdot tN,$$

\noindent
where the last equality comes from constraint (3). Likewise $b_i \leq b_i^*$. Thus

$$\max_{i \in [r]} |\{v \in V(K_{tN,tN}) : i \text{ touches } v\}| \leq \max_{i \in [r]} (a_i+b_i)tN \leq \max_{i \in [r]} (a_i^*+b_i^*)tN \leq mtN $$

\noindent
and so $g(tN,r) \leq mtN = g^*(r) \cdot tN$. But by Theorem 1 we also had $g(tN,r) \geq g^*(r) \cdot tN$ and so $g(tN,r) = g^*(r) \cdot tN$, as required. 
\end{proof}

\computeupper*

\begin{proof}
As in Theorem 2, let $m=g^*(r),a_i^*,b_i^*,a(R)^*,b(R)^*$ where $i \in [r]$ and $R \subseteq [r]$ be a rational solution to the coloring problem and let $N$ be the least common multiple of all the denominators of $a(R)^*,b(R)^*$ for all $R \subseteq [r]$. Let $C_r = 2N-2$.  Let $t = \lfloor \frac{n}{N} \rfloor$. As in Theorem 2, construct the coloring square $F : A_{tN} \times B_{tN} \rightarrow [r]$ satisfying $a_i \cdot tN +b_i \cdot tN \leq  g^*(r) \cdot tN$ for all $i \in [r]$. \\

We extend $F : A_{tN} \times B_{tN} \rightarrow [r]$ to $F' : A_{n} \times B_{n} \rightarrow [r]$ as follows. For $(i,j) \in [tN] \times [tN]$, define $F'(u_i,v_j) = F(u_i,v_j)$. Then for $(i,j) \in ([n] \setminus [tN]) \times [tN]$, define $F'(u_i,v_j) = F(u_{tN},v_j)$. Then for $(i,j) \in [n] \times ([n] \setminus [tN])$, define $F'(u_i,v_j) = F'(u_i,v_{tN})$. Visually, what we did was we first copied column $u_{tN}$ for $n-tN$ number of times to extend $F$ to the right, and then we copied the extended row $v_{tN}$ for $n-tN$ number of times to extend $F$ upwards. \\

Note when we extended $F$ to the right no color $i \in [r]$ touched any new row, so color $i$ can only touch more columns, namely the $n-tN \leq N-1$ columns we just added. Likewise, when we extended $F$ upwards no color $i \in [r]$ touched any new column, so color $i$ can only touch more rows, namely the $n-tN \leq N-1$ rows we just added. Thus for any color $i \in [r]$, the number of columns and rows it touches in $F'$ is upper bounded by 

$$a_i \cdot tN +b_i \cdot tN + 2(N-1)\leq g^*(r) \cdot tN + C_r \leq g^*(r) \cdot n + C_r.$$

\noindent
Thus $F'$ is a construction showing $g(n,r) \leq g^*(r) \cdot n + C_r$, as required.
\end{proof}

Theorems 1, 2, and 3 say that we can basically just study $g^*(r)$ instead of $g(n,r)$. This is what we'll be doing for the rest of the paper, and we will also replace $a_i^*$ with $a_i$, $b_i^*$ with $b_i$, $a(R)^*$ with $a(R)$, and $b(R)^*$ with $b(R)$ for notational convenience. In particular, given a solution that solves the coloring problem, for $i \in [r]$ we will have $a_i,b_i$ be the values of $a_i^*,b_i^*$ in that solution respectively, and for $R \subseteq [r]$ we will have $a(R), b(R)$ be the values of $a(R)^*, b(R)^*$ in that solution respectively. Note that functions $a,b : \mathcal{P}([r]) \rightarrow [0,1]$ satisfying constraint (3) determine the values of $a_i,b_i$ for all $i \in [r]$, so in order to specify a solution we just have to define $a(R),b(R)$ for all $R \subseteq [r]$. \\

\begin{center}
\section{$g^*(t^2) = \frac{2}{t}$}
\end{center}

In this Section we will show $g^*(t^2) = \frac{2}{t}$. We begin with a lemma. \\

\begin{lemma}
For any $a,b : \mathcal{P}([r]) \rightarrow [0,1]$ satisfying constraints (2), (3), and (5), we must have $\sum_{i \in [r]} a_ib_i \geq 1$.
\end{lemma}

\begin{proof}

\noindent
By constraint (3), we have

\begin{align*}
\sum_{i \in [r]} a_ib_i &= \sum_{i \in [r]} \left( \sum_{R \subseteq [r], \ i \in R} a(R) \right) \left( \sum_{R \subseteq [r], \ i \in R} b(R) \right) \\
&= \sum_{i \in [r]} \left( \sum_{R_1, R_2 \subseteq [r], \ i \in R_1 \cap R_2} a(R_1)a(R_2) \right).
\end{align*}

\noindent
By constraint (5), we have

\begin{align*}
\sum_{i \in [r]} \left( \sum_{R_1, R_2 \subseteq [r], \ i \in R_1 \cap R_2} a(R_1)a(R_2) \right) &\geq  \sum_{R_1, R_2 \subseteq [r] , \ R_1 \cap R_2 \neq \{\}} a(R_1) \cdot b(R_2) \\
&= \sum_{R_1, R_2 \subseteq [r]} a(R_1) \cdot b(R_2).
\end{align*}

\noindent
By constraint (2), we have

$$\sum_{R_1, R_2 \subseteq [r]} a(R_1) \cdot b(R_2) = \left( \sum_{R \subseteq [r]} a(R) \right) \left( \sum_{R \subseteq [r]} b(R) \right)  = 1$$

\noindent
and thus $\sum_{i \in [r]} a_ib_i \geq 1$ as required.
\end{proof}

Although we did not define coloring squares for a solution to the coloring problem, one can imagine something similar. Recall the proof of Theorem 2 where we constructed a coloring $f$ from a solution to the coloring problem. Similarly, we can construct a coloring $ \chi: [0,1]^2 \rightarrow [r]$ by partitioning the $x$-axis into $[0,1] = \bigsqcup_{R \subseteq [r]} A_R$ where $\mu (A_R) = a(R)$ and partitioning the $y$-axis into $[0,1] = \bigsqcup_{R \subseteq [r]} B_R$ where $\mu(B_R) = b(R)$ where $\mu$ is the Lebesgue measure and $A_R,B_R$ are intervals, which is possible because of constraints (1) and (2). Then for $x \in A_{R_1}$ and $y \in B_{R_2}$, we define $\chi (x,y) = i$ where $i \in R_1 \cap R_2$, which is possible since $R_1 \cap R_2 \neq \{\}$ because of constraint (5). So for $i \in [r]$, we note $a_ib_i$ is the area of the region that \textit{can} be covered by color $i$ (emphasis on \textit{can} because in constructing $\chi$ we had some choice over the colors in $R_1 \cap R_2$). For $\chi$ to be well-defined, everything in $[0,1]^2$ must be colored, and hence the total area of all the colors $\sum_{i \in [r]} a_ib_i$ must at least be the area of $[0,1]^2$, which is just $1$, and so  $\sum_{i \in [r]} a_ib_i \geq 1$. We will use Lemma 1 in proofs without citing it. \\

Also, given a coloring $\chi : [0,1]^2 \rightarrow [r]$ and projections $\pi_1,\pi_2 : [0,1]^2 \rightarrow [0,1]$ defined by $\pi_1(x,y) = x$ and $\pi_2(x,y) = y$, if $\{x \in [0,1] : \chi(\pi_1^{-1}(x)) = R\}$ and $\{y \in [0,1] : \chi(\pi_2^{-1}(y)) = R\}$ are  Lebesgue measurable for all $R \subseteq [r]$, define

$$a(R) = \mu(\{x \in [0,1] : \chi(\pi_1^{-1}(x)) = R\})$$

\noindent
and

$$b(R) = \mu(\{y \in [0,1] : \chi(\pi_2^{-1}(y)) = R\}).$$

\noindent
 By the properties of the Lebesgue measure, $a(R)$ and $b(R)$ satisfy the coloring constraints. In providing upper bounds for $g^*(r)$, in many cases it will be more convenient to describe a coloring of the unit square $\chi : [0,1]^2 \rightarrow [r]$ instead of specifying $a(R),b(R)$ for all $R \subseteq [r]$.  \\

Now we are in a position to prove $g^*(t^2) = \frac{2}{t}$. We first show a construction that exhbits $g^*(3^2) = g^*(9) \leq \frac{2}{3}$.

\begin{center}
\begin{tikzpicture}[declare function ={a=4;}]

\draw (0,0) coordinate (A) ;
\draw (a/3,0) coordinate (B) ;
\draw (2*a/3,0) coordinate (C) ;
\draw (a,0) coordinate (D) ;
\draw (0,a/3) coordinate (E) ;
\draw (a/3,a/3) coordinate (F) ;
\draw (2*a/3,a/3) coordinate (G) ;
\draw (a,a/3) coordinate (H) ;
\draw (0,2*a/3) coordinate (I) ;
\draw (a/3,2*a/3) coordinate (J) ;
\draw (2*a/3,2*a/3) coordinate (K) ;
\draw (a,2*a/3) coordinate (L) ;
\draw (0,a) coordinate (M) ;
\draw (a/3,a) coordinate (N) ;
\draw (2*a/3,a) coordinate (O) ;
\draw (a,a) coordinate (P) ;

\draw (A) -- (B) ;
\draw (B) -- (C) ;
\draw (C) -- (D) ;
\draw (E) -- (F) ;
\draw (F) -- (G) ;
\draw (G) -- (H) ;
\draw (I) -- (J) ;
\draw (J) -- (K) ;
\draw (K) -- (L) ;
\draw (M) -- (N) ;
\draw (N) -- (O) ;
\draw (O) -- (P) ;
\draw (A) -- (E) ;
\draw (B) -- (F) ;
\draw (C) -- (G) ;
\draw (D) -- (H) ;
\draw (E) -- (I) ;
\draw (F) -- (J) ;
\draw (G) -- (K) ;
\draw (H) -- (L) ;
\draw (I) -- (M) ;
\draw (J) -- (N) ;
\draw (K) -- (O) ;
\draw (L) -- (P) ;

\draw (a/6,a/6) node{\Large 7} ;
\draw (a/2,a/6) node{\Large 8} ;
\draw (5*a/6,a/6) node{\Large 9} ;
\draw (a/6,a/2) node{\Large 4} ;
\draw (a/2,a/2) node{\Large 5} ;
\draw (5*a/6,a/2) node{\Large 6} ;
\draw (a/6, 5*a/6) node{\Large 1} ;
\draw (a/2,5*a/6) node{\Large 2} ;
\draw (5*a/6, 5*a/6) node{\Large 3} ;

\draw[decorate, decoration={brace,raise=3pt,amplitude=4pt}] (A) -- (E) node[midway,xshift=-0.6cm]{\Large $\frac{1}{3}$}; 
\draw[decorate, decoration={brace,raise=3pt,amplitude=4pt}] (E) -- (I) node[midway,xshift=-0.6cm]{\Large $\frac{1}{3}$}; 
\draw[decorate, decoration={brace,raise=3pt,amplitude=4pt}] (I) -- (M) node[midway,xshift=-0.6cm]{\Large $\frac{1}{3}$}; 

\draw[decorate, decoration={brace,raise=3pt,amplitude=4pt,mirror}] (A) -- (B) node[midway,yshift=-0.6cm]{\Large $\frac{1}{3}$}; 
\draw[decorate, decoration={brace,raise=3pt,amplitude=4pt,mirror}] (B) -- (C) node[midway,yshift=-0.6cm]{\Large $\frac{1}{3}$}; 
\draw[decorate, decoration={brace,raise=3pt,amplitude=4pt,mirror}] (C) -- (D) node[midway,yshift=-0.6cm]{\Large $\frac{1}{3}$}; 

\draw (a/2, -2) node{\normalsize Figure 4} ;
\end{tikzpicture}
\end{center}

\noindent
This corresponds to the solution where $a(\{1,4,7\}) = a(\{2,5,8\}) = a(\{3,6,9\}) = b(\{1,2,3\}) = b(\{4,5,6\}) = b(\{7,8,9\}) = \frac{1}{3}$ and $a(R) = b(R) = 0$ for all other $R \subseteq [9]$. One can check $a_i + b_i = \frac{2}{3}$ for all $i \in [9]$. 

\squarelower*

\begin{proof}
Consider any solution $a,b : \mathcal{P}([r]) \rightarrow [0,1]$ satisfying constraints (1), (2), (3), and (5). Since $\sum_{i \in [r]} a_ib_i \geq 1$, for some $i \in [r]$ we have $a_ib_i \geq \frac{1}{r}$. This implies $a_i+b_i \geq \frac{2}{\sqrt{r}}$ by the AM-GM inequality, as required. 
\end{proof}

\square*

\begin{proof}
To see that $g^*(t^2) \leq \frac{2}{t}$, consider analogous constructions shown in Figure $4$, where we have $t^2$ squares arranged in a $t$ by $t$ grid, and each square gets a different color. Formally, for all $i \in [t]$ we have 

$$a(\{j \cdot r + i \in [t^2] : j \in \{0\} \cup [t-1]\}) = \frac{1}{t}$$ 

\noindent 
and for all $i \in \{0\} \cup [t-1]$

$$b(\{ i \cdot t + j \in [t^2] : j \in [t]\}) = \frac{1}{t}.$$

\noindent
And finally, $a(R) = b(R) = 0$ for other $R \subseteq [t^2]$. One can check that $a_i = b_i = \frac{2}{t}$ for all $i \in [t^2]$. \\

To see that $g^*(t^2) \geq \frac{2}{t}$, see Theorem 4. 
\end{proof}

\begin{center}
\section{$g^*(r)$ for Small Values of $r$}
\end{center}

In this Section, we determine the values of $g^*(r)$ for $2 \leq r \leq 8$. Although we have a program that can compute $g^*(r)$ for any value of $r$, it becomes intractable when $r \geq 4$, as recall it involves $2^{2^r} \times 2^{2^r}$ linear programs. This Section demonstrates techniques that can be used to determine values of $g^*(r)$ without having to run the program described in Section 3. Interspersed between determinations of $g^*(r)$ for various values of $r$, we will have lemmas which future determinations of $g^*(r)$ requires. In particular, a lemma will always directly precede a determination of $g^*(r)$ that requires it. \\

\begin{proposition}
We have $g^*(2) = \frac{3}{2}$.
\end{proposition}

\begin{proof}
To see that $g^*(2) \leq \frac{3}{2}$, consider \\

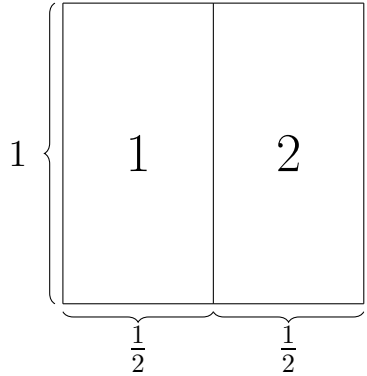
\begin{center}
\begin{tikzpicture}[declare function={a=4;}]

\draw (0,0) coordinate (A) ;
\draw (a/2,0) coordinate (B) ;
\draw (a,0) coordinate (C) ;
\draw (0,a) coordinate (D) ;
\draw (a/2,a) coordinate (E) ;
\draw (a,a) coordinate (F) ;

\draw (A) -- (B) ;
\draw (B) -- (C) ;
\draw (A) -- (D) ;
\draw (D) -- (E) ;
\draw (E) -- (F) ;
\draw (C) -- (F) ;
\draw (B) -- (E) ;

\draw (a/4,a/2) node{\huge 1} ;
\draw (3*a/4,a/2) node{\huge 2} ;

\draw[decorate, decoration={brace,raise=3pt,amplitude=4pt}] (A) -- (D) node[midway,xshift=-0.6cm]{\Large 1};
\draw[decorate, decoration={brace,mirror,raise=3pt,amplitude=4pt}] (A) -- (B) node[midway,yshift=-0.6cm]{\Large $\frac{1}{2}$} ;
\draw[decorate, decoration={brace,mirror,raise=3pt,amplitude=4pt}] (B) -- (C) node[midway,yshift=-0.6cm]{\Large $\frac{1}{2}$} ;
\end{tikzpicture}
\end{center}

Now we show $g^*(2) \geq \frac{3}{2}$. Consider any solution $a,b : \mathcal{P}([2]) \rightarrow [0,1]$ satisfying constraints (1), (2), (3), and (5). Without loss of generality, assume $a_1+b_1 \leq \frac{3}{2}$. We want to show $a_2+b_2 \geq \frac{3}{2}$. Note that $1-a_1 = a(\{2\})$ and $1-b_1 = b(\{2\})$ by constraint (5). Since $a_1+b_1 \leq \frac{3}{2}$, we have $a(\{2\}) + b(\{2\}) \geq \frac{1}{2}$. We consider two cases. \\

\textbf{Case 1}: $a(\{2\}),b(\{2\}) > 0$. Note $a(\{2\}) > 0$ implies that $b_2 = 1$ by constraint (5), and $b(\{2\}) > 0$ implies $a_2 = 1$ by constraint (5). Thus $a_2 + b_2 = 2 \geq \frac{3}{2}$. \\

\textbf{Case 2}: Without loss of generality assume $a(\{2\}) = 0$. Thus $b(\{2\}) \geq \frac{1}{2}$, which implies $a_2 = 1$. Thus $a_2 + b_2 \geq a_2 + b(\{2\}) \geq \frac{3}{2}$, as required. 
\end{proof}

\begin{lemma}
 For any solution $a,b : \mathcal{P}([r]) \rightarrow [0,1]$ satisfying constraints (1), (2), (3), and (5), we must have 
 
 $$\sum_{i \in [r]} a_i = \sum_{R \subseteq [r]} |R| \cdot a(R)$$
 $$\sum_{i \in [r]} b_i = \sum_{R \subseteq [r]} |R| \cdot b(R).$$
\end{lemma}

\begin{proof}
Observe that

$$\sum_{i \in [r]} a_i = \sum_{i \in [r]} \sum_{R \subseteq [r], i \in R} a(R) = \sum_{R \subseteq [r]} \sum_{i \in R} a(R) = \sum_{R \subseteq [r]} |R| \cdot a(R).$$

\noindent
The proof for $\sum_{i \in [r]} b_i = \sum_{R \subseteq [r]} |R| \cdot b(R)$ is analogous.
\end{proof}

A corollary of the previous lemma is that 

$$\sum_{i \in [r]} a_i + b_i = \sum_{R \subseteq [r]} |R| (a(R)+b(R)).$$

\noindent
In future proofs we will use this lemma and its corollary without citing it. \\

We make a couple of definitions that will be used in the proofs of the next propositions. We define the functions $a',b' : \mathcal{P}(\mathcal{P}([r])) \rightarrow [0,1]$.

\begin{definition}
For $P \subseteq \mathcal{P}([r])$, define $a'(P) = \sum_{R \in P} a(R)$ and $b'(P) = \sum_{R \in P} b(R)$.
\end{definition}

\begin{definition}
We define the upward closure $R^{\uparrow}$ of $R \subseteq [r]$ to be $R^{\uparrow} = \{T \in \mathcal{P}([r]) : T \supseteq R \}$.
\end{definition}

\noindent
Note that for $i \in [r]$ we have

$$a'(\{i\}^{\uparrow}) = \sum_{R \in \{i\}^{\uparrow}} a(R) = \sum_{R \subseteq [r], \ i \in R} a(R) = a_i.$$

\noindent
Likewise, $b'(\{i\}^{\uparrow}) = b_i$. \\

\begin{proposition}
We have $g^*(3) = \frac{5}{4}$.
\end{proposition}

\begin{proof}
To see that $g^*(3) \leq \frac{5}{4}$, consider \\

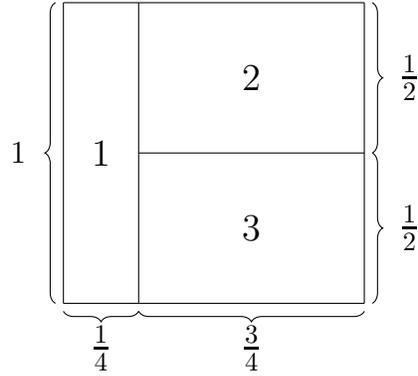
\begin{center}
\begin{tikzpicture}[declare function={a=4;}]

\draw (0,0) coordinate (A) ;
\draw (a/4,0) coordinate (B) ;
\draw (a,0) coordinate (C) ;
\draw (0,a) coordinate (D) ;
\draw (a/4,a) coordinate (E) ;
\draw (a,a) coordinate (F) ;
\draw (a/4,a/2) coordinate (G) ;
\draw (a,a/2) coordinate (H) ;

\draw (A) -- (B) ;
\draw (B) -- (C) ;
\draw (D) -- (E) ;
\draw (E) -- (F) ;
\draw (A) -- (D) ;
\draw (C) -- (F) ;
\draw (G) -- (H) ;
\draw (B) -- (E) ;

\draw (a/8,a/2) node{\Large 1} ;
\draw (5*a/8,3*a/4) node{\Large 2} ;
\draw (5*a/8,a/4) node{\Large 3} ;

\draw[decorate, decoration={brace,raise=3pt,amplitude=4pt}] (A) -- (D) node[midway,xshift=-0.6cm]{\large 1};
\draw[decorate, decoration={brace,raise=3pt,amplitude=4pt,mirror}] (A) -- (B) node[midway,yshift=-0.6cm]{\Large $\frac{1}{4}$} ;
\draw[decorate, decoration={brace,raise=3pt,amplitude=4pt,mirror}] (B) -- (C) node[midway,yshift=-0.6cm]{\Large $\frac{3}{4}$};
\draw[decorate, decoration={brace,raise=3pt,amplitude=4pt,mirror}] (C) -- (H) node[midway,xshift=0.6cm]{\Large $\frac{1}{2}$};
\draw[decorate, decoration={brace,raise=3pt,amplitude=4pt,mirror}] (H) -- (F) node[midway,xshift=0.6cm]{\Large $\frac{1}{2}$};

\end{tikzpicture}
\end{center}

Now we show $g^*(3) \geq \frac{5}{4}$. Consider any solution $a,b : \mathcal{P}([3]) \rightarrow [0,1]$ satisfying constraints (1), (2), (3), and (5). We consider two cases. \\

\textbf{Case 1}: For $R \subseteq [3]$, if $|R| \leq 1$, then $a(R) = b(R) = 0$. In other words if $a(R) > 0$ then $|R| \geq 2$ and if $b(R) > 0$ then $|R| \geq 2$. We note 

$$\sum_{i \in [3]} a_i + b_i = \sum_{R \subseteq [3]} |R|(a(R) + b(R)) \geq \sum_{R \subseteq [3]} 2(a(R) + b(R)) = 4$$

\noindent
where the last equality comes from constraint (2). Thus for some $i \in [3]$ we have $a_i + b_i \geq \frac{4}{3} \geq \frac{5}{4}$, and so we are done with this case. \\

\textbf{Case 2}: Without loss of generality assume $a(\{1\}) > 0$. This implies $b_1 = 1$ by constraint (5). Since we may assume $a_1 + b_1 \leq \frac{5}{4}$, we have that $a_1 \leq \frac{1}{4}$. This gives us

$$\sum_{R \subseteq \{2,3\}} a(R) = 1-a_1 \geq \frac{3}{4}$$

\noindent 
and

$$\sum_{R \in \{2\}^{\uparrow} \cup \{3\}^{\uparrow}} b(R) = 1.$$

Assume $a_2 + b_2 \leq \frac{5}{4}$. This implies then that 

$$a(\{3\}) + b'(\{3\}^{\uparrow} \setminus \{2\}^{\uparrow}) \geq \frac{3}{4} + 1 - \frac{5}{4} = \frac{1}{2}.$$

\noindent
We consider three subcases. One should visualize the rest of the proof as finding an optimal coloring of a $1-a_1$ by $1$ rectangle using the two colors $2,3$. We use an argument similar to the proof used in showing $g^*(2) = \frac{3}{2}$. \\

\textbf{Case 2.1}: $a(\{3\}), b'(\{3\}^{\uparrow} \setminus \{2\}^{\uparrow}) > 0$. Note $a(\{3\}) > 0$ implies that $b_3 = 1$ and  $b'(\{3\}^{\uparrow} \setminus \{2\}^{\uparrow}) > 0$ implies that $a_3 \geq 1 - a_1 \geq \frac{3}{4}$ and so $a_3+b_3 \geq 1 + \frac{3}{4} \geq \frac{5}{4}$. \\

\textbf{Case 2.2}: $a(\{3\}) = 0$. In this case we have $b'(\{3\}^{\uparrow} \setminus \{2\}^{\uparrow}) \geq \frac{1}{2}$. Note that this also implies $a_3 \geq 1-a_1 \geq \frac{3}{4}$ and thus $a_3 + b_3 \geq \frac{3}{4} + \frac{1}{2} = \frac{5}{4}$. \\

\textbf{Case 2.3}: $b'(\{3\}^{\uparrow} \setminus \{2\}^{\uparrow}) = 0$. In this case we have that $a(\{3\}) \geq \frac{1}{2}$, and thus $b_3 = 1$. So we have $a_3 + b_3 \geq \frac{1}{2} + 1 \geq \frac{5}{4}$, as required. 
\end{proof}

Note that if one fixes the sum of the side lengths of a rectangle $a_i + b_i = k$, to maximize its area one wants the rectangle to be a square with side lengths $a_i = b_i = \frac{k}{2}$. Conversely, the less like a square the rectangle is, i.e. $|a_i-b_i|$ is large, the less area it covers. Some conditions force color $i$ to occupy less area by forcing the value of $|a_i-b_i|$ to be large. Future lemmas will bound the sum of areas $\sum_{i \in S} a_ib_i$ some subset of colors $S \subseteq [r]$ can occupy given certain conditions on $a_i,b_i$. If the area they occupy is too small, then we can't satisfy $\sum_{i \in [r]} a_ib_i \geq 1$, and thus those certain conditions can't hold. These type of arguments put restrictions on what an optimal solution can look like. \\

\begin{lemma}
Let $c_i, d_i \in [0,1]$ where $i \in [s]$, $s \in \mathbb{Z}^{+}$, and $k \in (0,1]$. If $\sum_{i \in [s]} c_i \geq 1$ and $c_i + d_i \leq k \leq \frac{2}{s}$ for all $i \in [s]$, then $\sum_{i \in [s]} c_id_i \leq k - \frac{1}{s}$  
\end{lemma}

\begin{proof}
We will find an upper bound on the maximum of $\sum_{i \in [s]} c_id_i$ under the constraints above; the maximum of $\sum_{i \in [s]} c_id_i$ exists because the constrained space is compact and $\sum_{i \in [s]} c_id_i$ is continuous with respect to $c_i,d_i$. Note we may assume $c_i+d_i = k$ for all $i \in [s]$, since if $c_j + d_j < k$ for some $j \in [s]$ we can increase both of $c_j,d_j$ by a small amount, which will give us a larger value of $\sum_{i \in [s]} c_id_i$. Thus we now just need to find the maximum of $\sum_{i \in [s]} c_i(k-c_i)$ under the given constraints. \\

We now show we may also assume that $\sum_{i \in [s]} c_i = 1$. If we had $\sum_{i \in [s]} c_i > 1$, then for some $j \in [s]$ we have $c_j > \frac{1}{s} \geq \frac{k}{2}$. We may decrease $c_j$ by a small amount so that $c_j(k-c_j)$ increases which then would increase $f(c_1,c_2,\ldots,c_s) = \sum_{i \in [s]} c_i(k-c_i)$; the function $g(x) = x(k-x)$ has derivative $g'(x) = k - 2x$ and thus achieves a maximum at $\frac{k}{2}$, and furthermore is concave as $g''(x) = -2$, so moving $x$ closer to $\frac{k}{2}$ always increases $x(k-x)$. Under the constraint $\sum_{i \in [s]} c_i = 1$ the quantity we're trying to maximize $\sum_{i \in [s]} c_i(k-c_i)$ becomes 

$$\sum_{i \in [s]} c_i(k-c_i)= k \left( \sum_{i \in [s]} c_i \right) - \sum_{i \in [s]} c_i^2 = k - \sum_{i \in [s]} c_i^2.$$

We use Lagrange multipliers to find the maximum of  $f(c_1,c_2,\ldots,c_s) = k - \sum_{i \in [s]} c_i^2$ under the constraint $g_1(c_1,c_2,\ldots,c_s) = \sum_{i \in [s]} c_i = 1$. But first we argue that the maximum exists. Note we may assume that $c_i \geq 0$ for all $i \in [s]$, since $f(c_1,c_2,\ldots,c_s) = f(|c_1|,|c_2|,\ldots,|c_s|)$. The assumption $c_i \geq 0$ for all $i \in [s]$ along with the constraint $\sum_{i \in [s]} c_i = 1$ gives rise to a compact space, and since $f$ is continuous we know $f$ has a maximum in this compact space, and thus has a maximum without the assumption $c_i \geq 0$ for all $i \in [s]$. So $f(c_1,c_2,\ldots,c_s) = k - \sum_{i \in [s]} c_i^2$ has a maximum under the constraint $g_1(c_1,c_2,\ldots,c_s) = \sum_{i \in [s]} c_i = 1$. We now use Lagrange multipliers. \\

We note $\frac{\partial f}{\partial c_i} = - 2c_i$ for $i \in [s]$, and $\frac{\partial g_1}{\partial i} = 1$ for $i \in [s]$. Thus for some $\lambda \in \mathbb{R}$ we get

$$
\begin{pmatrix}
-2c_1 \\
-2c_2 \\
\vdots \\
-2c_s
\end{pmatrix} =
\lambda \begin{pmatrix}
1 \\
1 \\
\vdots \\
1
\end{pmatrix}
$$

\noindent
and so for all $i \in [s]$, we have $-2c_i = \lambda$. Thus we get for all $i,j \in [s]$, we have $-2c_i = -2c_j$, which implies $c_i = c_j$. Note

$$1 = \sum_{i \in [s]} c_i = \sum_{i \in [s]} c_1 = sc_1$$

\noindent
and so $c_1 = \frac{1}{s}$, which implies $c_i = \frac{1}{s}$ for all $i \in [s]$. Note 

$$f(\frac{1}{s},\frac{1}{s},\ldots,\frac{1}{s}) = \sum_{i \in [s]} \frac{1}{s}(k-\frac{1}{s}) = s \cdot \frac{1}{s}(k-\frac{1}{s}) = k-\frac{1}{s}$$

\noindent
as required. 
\end{proof}

\begin{lemma}
Let $c_i, d_i \in [0,1]$ where $i \in [s]$, $s \in \mathbb{Z}^{+}$, $k \in (0,1]$, and $A \in [0,1]$. If $c_i + d_i \leq k$ for all $i \in [s]$ and $\sum_{i \in [s]} c_id_i \geq A$, then we have 

$$s(\frac{k}{2} - \sqrt{\frac{k^2}{4}-\frac{A}{s}}) \leq \sum_{i \in [s]} \min\{c_i,d_i\} \leq \sum_{i \in [s]} \max\{c_i,d_i\} \leq s(\frac{k}{2} + \sqrt{\frac{k^2}{4}-\frac{A}{s}}).$$ 
\end{lemma}

\begin{proof}
We first show $\sum_{i \in [s]} \max\{c_i,d_i\} \leq s(\frac{k}{2} + \sqrt{\frac{k^2}{4}-\frac{A}{s}})$. Without loss of generality assume $c_i = \max\{c_i,d_i\}$ for all $i \in [s]$, so we need to show $\sum_{i \in [s]} c_i \leq s(\frac{k}{2} + \sqrt{\frac{k^2}{4}-\frac{A}{s}})$. We find an upper bound on the maximum of $f(c_1,c_2,\ldots,c_s) = \sum_{i \in [s]} c_i$ under the above constraints; the maximum exists because the constrained space is compact and $f$ is continuous with respect to $c_i$ and $d_i$. First note we may replace the constraint $c_i + d_i \leq k$ with $c_i + d_i = k$ for all $i \in [s]$. If $c_i+d_i < k$, we can increase $c_i$ so that $f(c_1,c_2,\ldots,c_s) = \sum_{i \in [s]} c_i$ attains a larger value while still maintaining the constraints. Thus the constraint $\sum_{i \in [s]} c_id_i \geq A$ becomes $\sum_{i \in [s]} c_i(k-c_i) \geq A$. \\

Next, we show that we may assume $\sum_{i \in [s]} c_i(k-c_i) = A$. If $\sum_{i \in [s]} c_i(k-c_i) > A$, for any $i \in [s]$ we may increase $c_i$ so that $f(c_1,c_2,\ldots,c_s) = \sum_{i \in [s]} c_i$ attains a larger value while still maintaining the constraints, unless $c_i = k$ for all $i \in [s]$, in which case the inequality $\sum_{i \in [s]} c_i(k-c_i) > A$ becomes $0 > A$, contradicting $A \in [0,1]$. Note under the constraint $\sum_{i \in [s]} c_i(k-c_i) = A$, we have

$$\sum_{i \in [s]} c_i(k-c_i) = k(\sum_{i \in [s]} c_i) - \sum_{i \in [s]} c_i^2 = k \cdot f(c_1,c_2,\ldots,c_s) - \sum_{i \in [s]} c_i^2 = A$$

\noindent
so that $f(c_1,c_2,\ldots,c_s) = \frac{A}{k} + \frac{1}{k}(\sum_{i \in [s]} c_i^2)$. \\

We argue that under the constraint $\sum_{i \in [s]} c_i(k-c_i) = A$, the function $f(c_1,c_2,\ldots,c_s) = \sum_{i \in [s]} c_i = \frac{A}{k} + \frac{1}{k}(\sum_{i \in [s]} c_i^2)$ has a maximum. Note that

$$f(c_1,c_2,\ldots,c_s) = f(|c_1|,|c_2|,\ldots,|c_s|)$$

\noindent
so that we may assume $c_i \geq 0$ for all $i \in [s]$. The assumption $c_i \geq 0$ for all $i \in [s]$ along with the constraint $\sum_{i \in [s]} c_i(k-c_i) = A$ give rise to a compact space; since $c_i(k-c_i) \leq \frac{k^2}{4}$ for all $i \in [s]$, if for any $j \in [s]$ we have $c_j$ becomes large enough so that $c_j(k-c_j) < -(s-1)\frac{k^2}{4}$, then $\sum_{i \in [s]} c_i(k-c_i) < 0 \leq A$. Since $f$ is continuous, it has a maximum in the compact space defined by $c_i \geq 0$ for all $i \in [s]$ and $\sum_{i \in [s]} c_i(k-c_i) = A$, and thus has a maximum with only the constraint $\sum_{i \in [s]} c_i(k-c_i) = A$. \\

We now use Lagrange multipliers to find the maximum of  $f(c_1,c_2,\ldots,c_s) = \sum_{i \in [s]} c_i$ under the constraint $\sum_{i \in [s]} c_i(k-c_i) = A$. We note $\frac{\partial f}{\partial c_i} = 1$ for all $i \in [s]$. Define $g(c_1,c_2,\ldots,c_s) = \sum_{i \in [s]} c_i(k-c_i)$, and so $\frac{\partial g}{\partial c_i} = k - 2c_i$ for all $i \in [s]$. Thus for some $\lambda \in \mathbb{R}$ we have

$$
\begin{pmatrix}
1 \\
1 \\
\vdots \\
1
\end{pmatrix} = 
\lambda 
\begin{pmatrix}
k-2c_1 \\
k-2c_2 \\
\vdots \\
k-2c_s
\end{pmatrix}
$$

\noindent
and so $1 = \lambda(k-2c_i)$ for all $i \in [s]$. Note we cannot have $\lambda = 0$, otherwise $1 = 0$, and so $k-2c_i = \frac{1}{\lambda}$ for all $i \in [s]$, which implies for all $i,j \in [s]$ we have $k-2c_i = k-2c_j$ so that $c_i = c_j$. Note

$$A = \sum_{i \in [s]} c_i(k-c_i) = \sum_{i \in [s]} c_1(k-c_1) = sc_1(k-c_1)$$

so that $c_1(k-c_1) = \frac{A}{s}$, which implies $c_1^2 - kc_1 = -\frac{A}{s}$, and thus $(c_1 - \frac{k}{2})^2 = \frac{k^2}{4} - \frac{A}{s}$ and so $c_1 = \frac{k}{2} \pm \sqrt{\frac{k^2}{4}-\frac{A}{s}}$. Note $c_1 =  \frac{k}{2} + \sqrt{\frac{k^2}{4}-\frac{A}{s}}$ is the larger of the two, and therefore

$$\sum_{i \in [s]} c_i = \sum_{i \in [s]} c_1 = s(\frac{k}{2} + \sqrt{\frac{k^2}{4}-\frac{A}{s}}).$$

Now we show $s(\frac{k}{2} - \sqrt{\frac{k^2}{4}-\frac{A}{s}}) \leq \sum_{i \in [s]} \min\{c_i,d_i\}$. Without loss of generality assume $c_i = \min\{c_i,d_i\}$ for all $i \in [s]$, so we need to show $\sum_{i \in [s]} c_i \geq s(\frac{k}{2} - \sqrt{\frac{k^2}{4}-\frac{A}{s}})$. We find a lower bound on the minimum of $f(c_1,c_2,\ldots,c_s) = \sum_{i \in [s]} c_i$ under the above constraints; the minimum exists because the constrained space is compact and $f$ is continuous with respect to $c_i$ and $d_i$. First note we may replace the constraint $c_i + d_i \leq k$ with $c_i + d_i = k$ for all $i \in [s]$. If for some $j \in [s]$ we had $c_j+d_j < k$, then we can increase $d_j$ so that $c_j + d_j = k$ without affecting the value of $f(c_1,c_2,\ldots,c_s) = \sum_{i \in [s]} c_i$ while still maintaining the constraints. Thus the constraint $\sum_{i \in [s]} c_id_i \geq A$ becomes $\sum_{i \in [s]} c_i(k-c_i) \geq A$. \\

Next, we show that we may assume $\sum_{i \in [s]} c_i(k-c_i) = A$. If $\sum_{i \in [s]} c_i(k-c_i) > A$, for any $i \in [s]$ we may decrease $c_i$ so that $f(c_1,c_2,\ldots,c_s) = \sum_{i \in [s]} c_i$ attains a smaller value while still maintaining the constraints, unless $c_i = 0$ for all $i \in [s]$ in which case the inequality $\sum_{i \in [s]} c_i(k-c_i) > A$ becomes $0 > A$, a contradiction. From this point on, the argument is analogous to the one showing $\sum_{i \in [s]} \max\{c_i,d_i\} \leq s(\frac{k}{2} + \sqrt{\frac{k^2}{4}-\frac{A}{s}})$.
\end{proof}

\begin{corollary}
Let $c_i,d_i \in [0,1]$ where $i \in [r]$ and $k \in (0,1]$. If $c_i+d_i \leq k$ for all $i \in [r]$ and $\sum_{i \in [r]} c_id_i \geq 1$. Then for $s \leq r, s \in \mathbb{Z}^{+}$ satisfying $ 1 - (r-s)\frac{k^2}{4} \in [0,1]$, we have

$$s(\frac{k}{2} - \sqrt{\frac{rk^2}{4s}-\frac{1}{s}}) \leq \sum_{i \in [s]} \min\{c_i,d_i\} \leq \sum_{i \in [s]} \max\{c_i,d_i\} \leq s(\frac{k}{2} + \sqrt{\frac{rk^2}{4s}-\frac{1}{s}}).$$
\end{corollary}

\begin{proof}
From the constraint $c_i+d_i \leq k$ for all $i \in [r]$, we have that $c_id_i \leq \frac{k^2}{4}$. Thus 

\begin{align*}
1 \leq \sum_{i \in [r]} c_id_i = \sum_{i \in [s]} c_id_i + \sum_{i \in [r] \setminus [s]} c_id_i &\leq \sum_{i \in [s]} c_id_i + \sum_{i \in [r] \setminus [s]} \frac{k^2}{4} \\
&= \sum_{i \in [s]} c_id_i +(r-s)\frac{k^2}{4}
\end{align*}

\noindent
so that $\sum_{i \in [s]} c_id_i \geq 1 - (r-s)\frac{k^2}{4}$. Apply Lemma 4 with $A = 1 - (r-s)\frac{k^2}{4}$, and we see that

$$\sqrt{\frac{k^2}{4}-\frac{A}{s}} = \sqrt{\frac{k^2}{4}-\frac{1}{s}(1 - (r-s)\frac{k^2}{4})} = \sqrt{\frac{rk^2}{4s}-\frac{1}{s}}.$$
\end{proof}

\begin{lemma}
For $t \in \mathbb{Z}^{+}$, if $a_i + b_i < \frac{1}{t}$ for $i \in [r]$, then for $|R| \leq t$ we have $a(R) = b(R) = 0$. 
\end{lemma}

\begin{proof}
Let $|R| \leq t$. Without loss of generality, if $a(R) > 0$, by constraint (5) this would imply

$$\sum_{i \in R} b_i \geq 1.$$

\noindent
Thus for some $i \in R$ we have $b_i \geq \frac{1}{|R|} = \frac{1}{t}$, and so $a_i + b_i \geq \frac{1}{t}$, a contradiction. \\
\end{proof}

\begin{proposition}
We have $g^*(5) = \frac{11}{12}$.
\end{proposition}

\begin{proof}
To see that $g^*(5) \leq \frac{11}{12}$, consider \\

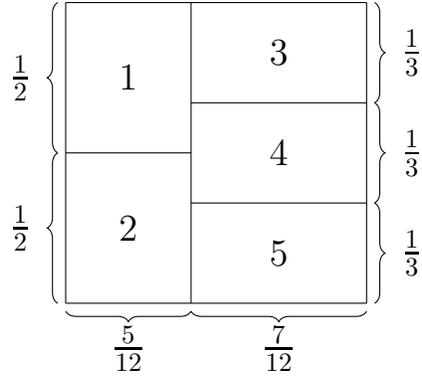
\begin{center}
\begin{tikzpicture}[declare function={a = 4;}]

\draw (0,0) coordinate (A) ;
\draw (5*a/12,0) coordinate (B) ;
\draw (a,0) coordinate (C) ;
\draw (0,a) coordinate (D) ;
\draw (5*a/12,a) coordinate (E) ;
\draw (a,a) coordinate (F) ;
\draw (0,a/2) coordinate (G) ;
\draw (5*a/12,a/2) coordinate (H) ;
\draw (5*a/12,2*a/3) coordinate (I) ;
\draw (a,2*a/3) coordinate (J) ;
\draw (5*a/12,a/3) coordinate (K) ;
\draw (a,a/3) coordinate (L) ;

\draw (A) -- (B) ;
\draw (B) -- (C) ;
\draw (D) -- (E) ;
\draw (E) -- (F) ;
\draw (A) -- (D) ;
\draw (C) -- (F) ;
\draw (B) -- (E) ;
\draw (G) -- (H) ;
\draw (I) -- (J) ;
\draw (K) -- (L) ;

\draw (5*a/24,3*a/4) node{\Large 1} ;
\draw (5*a/24,a/4) node{\Large 2} ;
\draw (17*a/24,5*a/6) node{\Large 3} ;
\draw (17*a/24,a/2) node{\Large 4} ;
\draw (17*a/24,a/6) node{\Large 5} ;

\draw[decorate, decoration={brace,raise=3pt,amplitude=4pt}] (A) -- (G) node[midway,xshift=-0.6cm]{\Large $\frac{1}{2}$} ;
\draw[decorate, decoration={brace,raise=3pt,amplitude=4pt}] (G) -- (D) node[midway,xshift=-0.6cm]{\Large $\frac{1}{2}$} ;
\draw[decorate, decoration={brace,raise=3pt,amplitude=4pt,mirror}] (A) -- (B) node[midway,yshift=-0.6cm]{\Large $\frac{5}{12}$} ;
\draw[decorate, decoration={brace,raise=3pt,amplitude=4pt,mirror}] (B) -- (C) node[midway,yshift=-0.6cm]{\Large $\frac{7}{12}$} ;
\draw[decorate, decoration={brace,raise=3pt,amplitude=4pt,mirror}] (C) -- (L) node[midway,xshift=0.6cm]{\Large $\frac{1}{3}$} ;
\draw[decorate, decoration={brace,raise=3pt,amplitude=4pt,mirror}] (L) -- (J) node[midway,xshift=0.6cm]{\Large $\frac{1}{3}$} ;
\draw[decorate, decoration={brace,raise=3pt,amplitude=4pt,mirror}] (J) -- (F) node[midway,xshift=0.6cm]{\Large $\frac{1}{3}$} ;

\end{tikzpicture}
\end{center}

We now show $g^*(5) \geq \frac{11}{12}$. Consider any solution $a,b : \mathcal{P}([5]) \rightarrow [0,1]$ satisfying constraints (1), (2), (3), and (5). We may assume $a_i + b_i \leq \frac{11}{12}$ for all $i \in [5]$. By Lemma 5, we note for $R \subseteq [5]$, if $|R| \leq 1$ then $a(R) = b(R) = 0$, since $a_i + b_i \leq \frac{11}{12} < \frac{1}{1}$ for all $i \in [5]$. We have

$$\sum_{R \subseteq [5]} |R| \cdot a(R) + |R| \cdot b(R) = \sum_{i \in [5]} a_i + b_i \leq 5 \cdot \frac{11}{12} = \frac{55}{12}.$$

\noindent
Without loss of generality we may thus assume $\sum_{R \subseteq [5]} |R| \cdot a(R)  \leq \frac{1}{2} \cdot \frac{55}{12} = \frac{55}{24}$. This gives us

\begin{align*}
\frac{55}{24} \geq \sum_{R \subseteq [5]} |R| \cdot a(R) &=  \left( \sum_{R \subseteq [5], |R| = 2} |R| \cdot a(R) \right) + \left( \sum_{R \subseteq [5], |R| \geq 3} |R| \cdot a(R) \right) \\
&\geq  \left( \sum_{R \subseteq [5], |R| = 2} 2 \cdot a(R) \right) + \left( \sum_{R \subseteq [5], |R| \geq 3} 3 \cdot a(R) 
\right) \\
&=  2 \left( \sum_{R \subseteq [5], |R| = 2} a(R) \right) + 3 \left( \sum_{R \subseteq [5], |R| \geq 3} a(R) 
\right) \\
&= 2 \left( \sum_{R \subseteq [5], |R| = 2} a(R) +  \sum_{R \subseteq [5], |R| \geq 3} a(R) \right) + \left( \sum_{R \subseteq [5], |R| \geq 3} a(R) 
\right) \\
&= 2 + \left( \sum_{R \subseteq [5], |R| \geq 3} a(R) \right). 
\end{align*}

\noindent
So we get $\frac{55}{24} \geq 2 + \sum_{R \subseteq [5], |R| \geq 3} a(R)$, which implies 

$$\sum_{R \subseteq [5], |R| \geq 3} a(R) \leq \frac{55}{24} - 2 = \frac{7}{24}.$$

\noindent 
This implies

\begin{equation}
\sum_{R \subseteq [5], |R| = 2} a(R) \geq 1- \frac{7}{24} = \frac{17}{24}
\end{equation}

Recall we have that $\sum_{R \subseteq [5]} |R| \cdot a(R) + |R| \cdot b(R) \leq \frac{55}{12}$ and we also have 

$$\sum_{R \subseteq [5]} |R| \cdot a(R) = \sum_{R \subseteq [5], |R| \geq 2} |R| \cdot a(R) \geq 2$$ 

\noindent
thus 

$$\sum_{R \subseteq [5]} |R| \cdot b(R) \leq \frac{55}{12} - 2 = \frac{31}{12}.$$

\noindent
Thus we have 

\begin{align*}
 \frac{31}{12} \geq \sum_{R \subseteq [5]} |R| \cdot b(R) &= \left( \sum_{R \subseteq [5], |R| = 2} |R| \cdot b(R) \right) + \left( \sum_{R \subseteq [5], |R| \geq 3} |R| \cdot b(R) \right) \\
&\geq 2 \left( \sum_{R \subseteq [5], |R| = 2} b(R) \right) + 3 \left( \sum_{R \subseteq [5], |R| \geq 3} b(R) \right) \\
&= 2 + \left( \sum_{R \subseteq [5], \ |R| \geq 3} b(R) \right)
\end{align*}

\noindent
and so $\frac{31}{12} \geq 2 + \sum_{R \subseteq [5], |R| \geq 3} b(R)$, thus

$$\sum_{R \subseteq [5], \ |R| \geq 3} b(R) \leq \frac{31}{12} - 2 = \frac{7}{12}$$

\noindent
which in turn implies

\begin{equation}
\sum_{R \subseteq [5],|R| = 2} b(R) \geq \frac{5}{12}.
\end{equation}

\textbf{Claim 1}: For $i,j,k,l$ distinct, we can't have both $b(\{i,j\}) >0$ and $b(\{k,l\}) > 0$. Assume for the sake of contradiction otherwise, and without loss of generality assume $i = 1, j = 2, k = 3, l = 4$ so that $b(\{1,2\}) > 0$ and $b(\{3,4\}) > 0$. This implies that $a_1 + a_2 \geq 1$ and $a_3 + a_4 \geq 1$ respectively. Without loss of generality, assume $a_1 \geq \frac{1}{2}$, so that since $\sum_{R \subseteq [5], |R| = 2} a(R) \geq \frac{17}{24} > 1 - \frac{1}{2}$ from (7), we must have $a(\{1,i\}) > 0$ for some $i \in [5] \setminus \{1\}$. Furthermore since $b(\{3,4\}) > 0$ and so $a'(\{3\}^{\uparrow} \cup \{4\}^{\uparrow}) = 1$, we may assume without loss of generality $i = 3$ so that $a(\{1,3\}) > 0$. This implies that $b_1 + b_3 \geq 1$. Since $a_i + b_i \leq \frac{11}{12}$ for all $i \in [5]$ and $a_1 \geq \frac{1}{2}$, we must have that $b_1 \leq \frac{5}{12}$ so that $b_3 \geq \frac{7}{12}$, which in turn implies $a_3 \leq \frac{1}{3}$. Since $a_3 \leq \frac{1}{3}$ and $a_3 + a_4 \geq 1$, we must have $a_4 \geq \frac{2}{3}$ and so $b_4 \leq \frac{1}{4}$. \\

Now we note with the constraints $a_1 + a_2 \geq 1$ and $a_i+b_i \leq \frac{11}{12}$ for $i \in [2]$, we have $a_1b_1 + a_2b_2 \leq \frac{5}{12}$ by Lemma 3 as $\frac{11}{12} \leq \frac{2}{2} 
= 1$. Since $a_3 \leq \frac{1}{3}, b_3 \geq \frac{7}{12}$ we have $a_3b_3 \leq \frac{7}{36}$. Since $a_4 \geq \frac{2}{3}$ and $b_4 \leq \frac{1}{4}$, we have $a_4b_4 \leq \frac{1}{6}$. Thus we have

$$a_1b_1 + a_2b_2 + a_3b_3 + a_4b_4 \leq \frac{5}{12} + \frac{7}{36} + \frac{1}{6} = \frac{28}{36} = \frac{7}{9}.$$

\noindent
This implies $a_5b_5 \geq 1 - \frac{7}{9} = \frac{2}{9}$. But since $a_5 + b_5 \leq \frac{11}{12}$, we have $a_5b_5 \leq (\frac{1}{2} \cdot \frac{11}{12})^2 < \frac{2}{9}$, a contradiction. \\

\textbf{Claim 2}: We show we cannot have both $b(\{i,j\}) > 0$ and $b(\{i,k\}) > 0$ for $i,j,k \in [5]$ distinct. Assume otherwise for the sake of contradiction, and without loss of generality assume $i=1,j=2,k=3$, so that $b(\{1,2\}) > 0$ and $b(\{1,3\}) > 0$. This implies that $a_1 + a_2 \geq 1$ and $a_1 + a_3 \geq 1$ respectively. By Corollary 1, $a_1 \leq 0.69$. This implies $a'(\{2,3\}^{\uparrow} \setminus \{1\}^{\uparrow}) \geq 1-0.69 = 0.31$ because $b(\{1,2\}) > 0$ and $b(\{1,3\}) > 0$ which implies if $a(R) > 0$ and $1 \not \in R$ then $2,3 \in R$. Recall $\sum_{R \subseteq [5], |R| = 2} a(R) \geq \frac{17}{24}$ from (7), and we note that $\frac{17}{24} > 1 - 0.31$, so that $a(\{2,3\}) > 0$, and so $b_2 + b_3 \geq 1$. \\

We now show that $\{1,2\}, \{1,3\}$ are the only two subsets $R \subseteq [5]$ that satisfy $|R| = 2$ and and $b(R) > 0$. Consider some $R \subseteq [5]$ such that $b(R) > 0$ and $|R| = 2$. Without loss of generality, we can assume $2 \in R$ because $a(\{2,3\}) > 0$. If $1 \in R$ we just have $R = \{1,2\}$, so we are done. If $3 \in R$, then we have $b(\{2,3\}) > 0$ and so $a_2 + a_3 \geq 1$. But recall we also have $b_2 + b_3 \geq 1$, so that $(a_2 + b_2) + (a_3 + b_3) \geq 2$, so that either $a_2 + b_2 \geq 1$ or $a_3 + b_3 \geq 1$, contradicting $a_i + b_i \leq \frac{11}{12}$ for all $i \in [5]$. Thus without loss of generality let $4 \in R$, so $R =  \{2,4\}$. So we have $b(\{2,4\}) > 0$, but we also have $b(\{1,3\}) > 0$ by assumption, contradicting Claim 1. \\

The previous paragraph shows that if $R \neq \{1,2\}, \{1,3\}$ and $b(R) > 0$, then $|R| \geq 3$. Recall equation (8)

$$\sum_{R \subseteq [5], |R| = 2} b(R) = b(\{1,2\}) + b(\{1,3\}) \geq \frac{5}{12}.$$

\noindent
This gives us $b_1 \geq \frac{5}{12}$ so that $a_1 \leq \frac{1}{2}$. Since $a_1 + a_2 \geq 1$ and $a_1 + a_3 \geq 1$, this implies $a_2,a_3 \geq \frac{1}{2}$, and so $b_2,b_3 \leq \frac{5}{12}$ and thus $b_2 + b_3 \leq \frac{5}{6}$, contradicting $b_2 + b_3 \geq 1$. \\

We now finally show $g^*(5) \geq \frac{11}{12}$. From (8), we know there exists $|R| = 2$ such that $b(R) > 0$, so without loss of generality assume $b(\{1,2\}) > 0 $. This implies $a_1 + a_2 \geq 1$. This in turn implies $b_1 + b_2 \leq 2 \cdot \frac{11}{12} - 1 = \frac{5}{6}$. In particular, $b'(\{1\}^{\uparrow} \cup \{2\}^{\uparrow}) \leq \frac{5}{6}$, which implies $\sum_{R \subseteq [5] \setminus [2]} b(R) \geq \frac{1}{6}$. We claim $\sum_{R \subseteq [5] \setminus [2]} b(R) = b(\{3,4,5\})$ so that $b(\{3,4,5\}) \geq \frac{1}{6}$. Note $b(R) > 0$ implies $|R| \geq 2$, and if $|R| = 2$ and $R \subseteq [5] \setminus [2]$, this would contradict Claim 1 as we also have $b(\{1,2\}) > 0$. Thus $|R| = 3$ and so $R = \{3,4,5\}$.  \\

We claim that if $b(R) > 0$ and $|R| = 2$, then $R = \{1,2\}$. So assume $b(R) > 0$ and $|R| = 2$. First assume $1 \in R$. If $2 \in R$ we are done. If $3 \in R$, we'd have $R = \{1,3\}$, contradicting claim 2 because $b(\{1,2\}) > 0$ as well. Likewise if $4,5 \in R$. If we start off with the assumption $2 \in R$, we get a symmetric argument. Thus $R$ cannot contain $1,2$. This means $R \subseteq [5] \setminus [2]$, so by the preceding paragraph we have $R = \{3,4,5\}$, and so $|R| \neq 2$. \\

Recall equation (8) that $\sum_{R \subseteq [5], |R| = 2} b(R) = b(\{1,2\}) \geq \frac{5}{12}$. So $b_1, b_2 \geq \frac{5}{12}$. Recall we also had $a_1+a_2 \geq 1$. Thus $(a_1+b_1)+(a_2+b_2) \geq 1 + 2 \cdot \frac{5}{12} = \frac{11}{6}$, so for some $i \in [2]$ we have $a_i+b_i \geq \frac{11}{12}$, as required. 
\end{proof}

\begin{lemma}
Let $c_i, d_i \in [0,1]$ where $i \in [s]$, $k \in (0,1]$, and $[s] = I \sqcup O_1 \sqcup O_2$ where $|I| \geq 1$ and $|O_1| = |O_2| \geq 1$. If $c_i + d_i \leq k < \frac{2}{|I|+|O_1|}$ for all $i \in [s]$ and 

$$\sum_{i \in I \sqcup O_1} c_i \geq 1$$ 
$$\sum_{i \in I \sqcup O_2} d_i \geq 1,$$

\noindent
then if $e = \frac{2-|I| \cdot k}{2|O_1|}$ we have 

$$\sum_{i \in [s]} c_id_i \leq |I|\frac{k^2}{4} + 2|O_1| \cdot e(k-e).$$ 
\end{lemma}

\begin{proof}
We find an upper bound on the maximum of $\sum_{i \in [s]} c_id_i$ under the constraints above. The maximum exists because $\sum_{i \in [s]} c_id_i$ is continuous and the constrained space is compact with respect to $c_i,d_i$. We may assume $c_i + d_i = k$ for all $i \in [s]$, because if $c_j + d_j < k$ for some $j \in [s]$, then we may increase both $c_j,d_j$ so that $\sum_{i \in [s]} c_id_i$ increases. Thus we can just consider maximizing  

$$f(c_1,c_2,\ldots,c_{s}) = \sum_{i \in [s]} c_i(k-c_i)$$

\noindent
and the constraint $\sum_{i \in I \sqcup O_2} d_i \geq 1$ becomes

$$\sum_{i \in I \sqcup O_2} k-c_i  \geq 1.$$

We now show we can assume $\sum_{i \in I \sqcup O_1} c_i  = 1$ and $\sum_{i \in I \sqcup O_2} k-c_i  = 1$. We first note that the derivative of $g(x) = x(k-x)$ is $g'(x) = k-2x$, and so $g(x)$ is maximized at $x = \frac{k}{2}$. Furthermore, $g''(x) = -2$ so that if $x$ approaches $\frac{k}{2}$ we know $g(x)$ increases. If we had $\sum_{i \in I \sqcup O_1} c_i  > 1$, this would imply for some $j \in I \sqcup O_1$ we had $c_j > \frac{1}{|I| + |O_1|} > \frac{k}{2}$, where the last inequality comes from $k < \frac{2}{|I|+|O_1|}$. Thus we may decrease $c_j$ so that $f(c_1,c_2,\ldots,c_{s}) = \sum_{i \in [s]} c_i(k-c_i)$ attains a larger value while still maintaining the constraints. So let $\sum_{i \in I \sqcup O_1} c_i  = 1$ be the new constraint in place of $\sum_{i \in I \sqcup O_1} c_i  \geq 1$. \\

Now we show we may assume $\sum_{i \in I \sqcup O_2} k-c_i = 1$. So assume otherwise that $\sum_{i \in I \sqcup O_2} k-c_i > 1$. Note if for some $j \in O_2$ we had $k -  c_j > \frac{k}{2}$, then we may increase $c_j$ so that $f(c_1,c_2,\ldots,c_{s}) = \sum_{i \in [s]} c_i(k-c_i)$ attains a larger value while still maintaining the constraints. So assume for all $i \in O_2$, we have $k-c_i \leq \frac{k}{2}$. If for some $j \in O_2$ we had $k-c_j < \frac{k}{2}$, then we may decrease $c_j$ so that $f(c_1,c_2,\ldots,c_{s}) = \sum_{i \in [s]} c_i(k-c_i)$ attains a larger value while still maintaining the constraints. So assume $k-c_i = \frac{k}{2}$ for all $i \in O_2$. Thus $\sum_{i \in O_2} k-c_i = |O_2|\frac{k}{2}$, so along with $\sum_{i \in I \sqcup O_2} k-c_i > 1$ we get $\sum_{i \in I} k-c_i > 1 - | O_2|\frac{k}{2}$, and so 

$$\sum_{i \in I} c_i < |I|k + |O_2|\frac{k}{2} - 1 = (2|I| + |O_2|)\frac{k}{2}-1. $$ 

Since $\sum_{i \in I \sqcup O_1} c_i = 1$ and $\sum_{i \in I} c_i < (2|I| + |O_2|)\frac{k}{2}-1$, we must have 

$$\sum_{i \in O_1} c_i = 1-\sum_{i \in I} c_i > 2-(2|I| + |O_2|)\frac{k}{2}$$

\noindent
and thus for some $j \in O_1$ we have $c_j > \frac{2}{|O_1|} - \frac{k(2|I|+|O_1|)}{2|O_1|} > \frac{k}{2}$, where the last inequality can be rearranged into $\frac{k}{2} < \frac{1}{|I|+|O_1|}$. From $\sum_{i \in I} c_i <  (2|I| + |O_2|)\frac{k}{2}-1 $ we also get for some $l \in I$ we have $c_l < (2+\frac{|O_2|}{|I|})\frac{k}{2} - \frac{1}{|I|} < \frac{k}{2}$ where the last inequality can also be rearranged into $\frac{k}{2} < \frac{1}{|I|+|O_1|}$. \\

Note we can increase $c_l$ by $\epsilon > 0$ and decrease $c_j$ by $\epsilon$ so that $f(c_1,c_2,\ldots,c_{s}) = \sum_{i \in [s]} c_i(k-c_i)$ attains a larger value while still maintaining the constraints $\sum_{i \in I \sqcup O_1} c_i = 1$ and $\sum_{i \in I \sqcup O_2 } k-c_i \geq 1$. The constraint $\sum_{i \in I \sqcup O_1} c_i = 1$ is preserved because the $\epsilon$'s cancel out. The constraint $\sum_{i \in I \sqcup O_2 } k-c_i \geq 1$ can be preserved for small enough $\epsilon$ because by assumption $\sum_{i \in I \sqcup O_2 } k-c_i > 1$. Thus under the assumption $\sum_{i \in I \sqcup O_2 } k-c_i > 1$, we can always find a larger value of $f(c_1,c_2,\ldots,c_{s}) = \sum_{i \in [s]} c_i(k-c_i)$ while maintaining the desired constraints, so we may assume $\sum_{i \in I \sqcup O_2} k-c_i = 1$. \\

We claim under the constraints $\sum_{i \in I \sqcup O_1} c_i = 1$ and $\sum_{i \in I \sqcup O_2 } k-c_i = 1$, the function $f(c_1,c_2,\ldots,c_{s}) = \sum_{i \in [s]} c_i(k-c_i)$ has a maximum. From the constraint $\sum_{i \in I \sqcup O_2 } k-c_i = 1$ we get that $\sum_{i \in I \sqcup O_2} c_i = (|I| + |O_2|)k-1$. Note we have

\begin{align*}
f(c_1,c_2,\ldots,c_{s}) &= \sum_{i \in [s]} c_i(k-c_i) = k \left( \sum_{i \in [s]} c_i \right) - \sum_{i \in [s]} c_i^2 \\
&= k \left( \sum_{i \in I \sqcup O_1} c_i + \sum_{i \in O_2} c_i \right) - \sum_{i \in [s]} c_i^2 = k \left( 1+ \sum_{i \in O_2} c_i \right) - \sum_{i \in [s]} c_i^2.
\end{align*}

\noindent
This shows for any $i \in I \sqcup O_1$, we have 

$$f(c_1,c_2, \ldots, c_i, \ldots, c_s) = f(c_1,c_2, \ldots, |c_i|, \ldots, c_s)$$

\noindent
and analogously 

\begin{align*}
f(c_1,c_2,\ldots,c_{s}) &= \sum_{i \in [s]} c_i(k-c_i) = k \left( \sum_{i \in [s]} c_i \right) - \sum_{i \in [s]} c_i^2 \\
&= k \left( \sum_{i \in O_1} c_i + \sum_{i \in I \sqcup O_2} c_i \right) - \sum_{i \in [s]} c_i^2 \\
&= k \left( \left( \sum_{i \in O_1} c_i \right) +  (|I| + |O_2|)k-1  \right) - \sum_{i \in [s]} c_i^2
\end{align*}

\noindent 
which shows for any $i \in I \sqcup O_2$, we have 

$$f(c_1,c_2, \ldots, c_i, \ldots, c_s) = f(c_1,c_2, \ldots, |c_i|, \ldots, c_s).$$ 

\noindent
Thus $f(c_1,c_2,\ldots,c_s) = f(|c_1|,|c_2|,\ldots,|c_s|)$, so we may assume $c_i \geq 0$ for all $i \in [s]$. Note along with the assumption $c_i \geq 0$ for all $i \in [s]$, the constraints $\sum_{i \in I \sqcup O_1} c_i = 1$ and $\sum_{i \in I \sqcup O_2 } k-c_i = 1$ give rise to a compact space, and since $f$ is continuous we know $f$ attains a maximum in this compact space defined by the assumption $c_i \geq 0$ for all $i \in [s]$ with the constraints $\sum_{i \in I \sqcup O_1} c_i = 1$ and $\sum_{i \in I \sqcup O_2 } k-c_i = 1$, and thus $f$ attains a maximum with the constraints $\sum_{i \in I \sqcup O_1} c_i = 1$ and $\sum_{i \in I \sqcup O_2 } k-c_i = 1$ only. \\

We use Lagrange multipliers to maximize $f(c_1,c_2,\ldots,c_{s}) = \sum_{i \in [s]} c_i(k-c_i)$ under the constraints $\sum_{i \in I \sqcup O_1} c_i = 1$ and $\sum_{i \in I \sqcup O_2 } k-c_i = 1$. We have $\frac{\partial f}{\partial c_i} = k - 2c_i$ for all $i \in [s]$. Define $g_1(c_1,c_2,\ldots,c_s) = \sum_{i \in I \sqcup O_1} c_i$ and $g_2(c_1,c_2,\ldots,c_s) = -(\sum_{i \in I \sqcup O_2 } k-c_i)$. So $\frac{\partial g_1}{\partial c_i} = 1$ if $i \in I \cup O_1$ and $0$ otherwise, and $\frac{\partial g_2}{\partial c_i} = 1$ if $i \in I \cup O_2$ and $0$ otherwise. Thus for some $\lambda_1,\lambda_2 \in \mathbb{R}$ we have 
$k-2c_i = \lambda_1+\lambda_2$ for all $i \in I$, and $k-2c_i = \lambda_1$ for all $i \in O_1$, and $k-2c_i = \lambda_2$ for all $i \in O_2$. This gives us for all $i,j \in I$ we have $c_i = c_j$, for all $i,j \in O_1$ we have $c_i = c_j$, and for all $i,j \in O_2$ we have $c_i = c_j$. \\

Without loss of generality, let $c_1 \in O_1, c_2 \in O_2, c_3 \in I$. We get 

$$\sum_{i \in I \sqcup O_1} c_i = \sum_{i \in I} c_i + \sum_{i \in O_1} c_i = \sum_{i \in I} c_3 + \sum_{i \in O_1} c_1 = |I|c_3 + |O_1|c_1 = 1.$$

\noindent
Likewise

$$\sum_{i \in I \sqcup O_2} k-c_i = \sum_{i \in I}k- c_i + \sum_{i \in O_2}k- c_i = |I|(k-c_3) + |O_2|(k-c_2) = 1.$$

\noindent
Thus we get $|I|c_3 + |O_1|c_1 =|I|(k-c_3) + |O_2|(k-c_2)$, and so 

$$2|I|c_3 = |I|k+|O_1|(k-c_2-c_1)$$

\noindent
which gives us $2c_3 = k+ \frac{|O_1|}{|I|}(k-c_2-c_1)$. Recall also that $k-2c_3 = \lambda_1 + \lambda_2 = k-2c_1+k-2c_2 $ and thus $2c_3 = 2c_1+2c_2-k$. Thus we have 

$$2c_1+2c_2-k = 2c_3 = k+\frac{|O_1|}{|I|}(k-c_2-c_1)$$

\noindent
which rearranges into

$$0 = (\frac{|O_1|}{|I|} + 2)(k-c_2-c_1)$$

\noindent
and thus $c_1+c_2 = k$. \\

Recall $2c_3 = 2c_1+2c_2 - k$, so along with $c_1+c_2 = k$ we get $2c_3 = k$ so that $c_3 = \frac{k}{2}$. Recall also that $|I|c_3 + |O_1|c_1 = 1$, so that 

$$c_1 = \frac{1}{|O_1|}(1-|I|\frac{k}{2}) = \frac{2-|I|k}{2|O_1|} = e.$$

\noindent
Note from $c_1+c_2 = k$ we get $c_1(k-c_1) = c_2(k-c_2)$ Finally, we compute

\begin{align*}
f(c_1,c_2,\ldots,c_s) &= |I|c_3(k-c_3) + |O_1|c_1(k-c_1)+|O_2|c_2(k-c_2) \\
&= |I|\frac{k^2}{4} + 2|O_1| \cdot e(k-e) 
\end{align*}

\noindent
as required.
\end{proof}

\begin{proposition}
We have $g^*(6) = \frac{5}{6}$.
\end{proposition}

\begin{proof}
To see that $g^*(6) \leq \frac{5}{6}$, consider \\

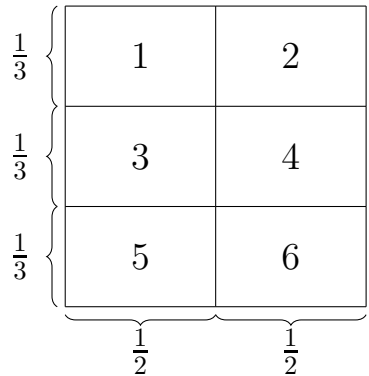
\begin{center}
\begin{tikzpicture}[declare function={a=4;}]
\draw (0,0) coordinate (A) ;
\draw (a/2,0) coordinate (B) ;
\draw (a,0) coordinate (C) ;
\draw (0,a/3) coordinate (D) ;
\draw (a/2,a/3) coordinate (E) ;
\draw (a,a/3) coordinate (F) ;
\draw (0,2*a/3) coordinate (G) ;
\draw (a/2,2*a/3) coordinate (H) ;
\draw (a,2*a/3) coordinate (I) ;
\draw (0,a) coordinate (J) ;
\draw (a/2,a) coordinate (K) ;
\draw (a,a) coordinate (L) ;

\draw (A) -- (B) ;
\draw (B) -- (C) ;
\draw (D) -- (E) ;
\draw (E) -- (F) ;
\draw (G) -- (H) ;
\draw (H) -- (I) ;
\draw (J) -- (K) ;
\draw (K) -- (L) ;
\draw (A) -- (D) ;
\draw (B) -- (E) ;
\draw (C) -- (F) ;
\draw (D) -- (G) ;
\draw (E) -- (H) ;
\draw (F) -- (I) ;
\draw (G) -- (J) ;
\draw (H) -- (K) ;
\draw (I) -- (L) ;

\draw (a/4,5*a/6) node{\Large 1} ;
\draw (3*a/4,5*a/6) node{\Large 2} ;
\draw (a/4,a/2) node{\Large 3} ;
\draw (3*a/4,a/2) node{\Large 4} ;
\draw (a/4,a/6) node{\Large 5} ;
\draw (3*a/4,a/6) node{\Large 6} ;

\draw[decorate,decoration={brace,raise=3pt,amplitude=4pt}] (A) -- (D) node[midway,xshift=-0.6cm]{\Large $\frac{1}{3}$} ;
\draw[decorate,decoration={brace,raise=3pt,amplitude=4pt}] (D) -- (G) node[midway,xshift=-0.6cm]{\Large $\frac{1}{3}$} ;
\draw[decorate,decoration={brace,raise=3pt,amplitude=4pt}] (G) -- (J) node[midway,xshift=-0.6cm]{\Large $\frac{1}{3}$} ;
\draw[decorate,decoration={brace,raise=3pt,amplitude=4pt,mirror}] (A) -- (B) node[midway,yshift=-0.6cm]{\Large $\frac{1}{2}$} ;
\draw[decorate,decoration={brace,raise=3pt,amplitude=4pt,mirror}] (B) --(C) node[midway,yshift=-0.6cm]{\Large $\frac{1}{2}$} ;
\end{tikzpicture}
\end{center}

Now we show $g^*(6) \geq \frac{5}{6}$. Consider any solution $a,b : \mathcal{P}([6]) \rightarrow [0,1]$ satisfying constraints (1), (2), (3), and (5). We may assume $a_i+b_i \leq \frac{5}{6}$ for all $i \in [6]$. By Lemma 5, if $|R| \leq 1$ then $a(R) = b(R) = 0$, as $a_i  + b_i \leq \frac{5}{6} < \frac{1}{1}$ for all $i \in [6]$. \\

We claim there do not exist $R_1,R_2 \subseteq [6]$ such that $a(R_1), b(R_2) > 0$ and $|R_1|=|R_2| = 2$. Assume otherwise for the sake of contradiction. By constraint (5), we have $R_1 \cap R_2 \neq \{\}$, so we may without loss of generality assume $1 \in R_1 \cap R_2$. If without loss of generality $R_1 = R_2 = \{1,2\}$, then we have $a(\{1,2\}),b(\{1,2\}) > 0$ which implies $b_1+b_2 \geq 1$ and $a_1+a_2 \geq 1$ respectively, so that $(a_1+b_1) + (a_2+b_2) \geq 2$ and thus for some $i \in [2]$ we have $a_i+b_i \geq 1$, contradicting $a_i+b_i \leq \frac{5}{6}$. Now without loss of generality assume $R_1 = \{1,2\}$ and $R_2 = \{1,3\}$ so that $a(\{1,2\}) > 0$ and $b(\{1,3\}) > 0$ which implies $b_1 + b_2 \geq 1$ and $a_1 + a_3 \geq 1$ respectively. Apply Lemma 6 with $I = \{1\}, O_1 = \{2\}, O_2 = \{3\}$ and $k = \frac{5}{6}$. We just need to check $\frac{5}{6} < \frac{2}{1+1}$, which is true. Thus Lemma 6 says that $a_1b_1 + a_2b_2 + a_3b_3 \leq 0.47$. Thus we have $ 3 \cdot (\frac{1}{2} \cdot \frac{5}{6})^2 \geq a_4b_4 + a_5b_5 + a_6b_6 \geq 1-0.47 = 0.53$, a contradiction. \\

Thus we may without loss of generality assume if $a(R) > 0$, then $|R| \geq 3$. Thus

\begin{align*}
\sum_{i \in [6]} a_i + b_i &= \sum_{R \subseteq [r]} |R| \cdot a(R) + |R| \cdot b(R) \\
&= \left( \sum_{R \subseteq [r], |R| \geq 3} |R| \cdot a(R) \right) + \left( \sum_{R \subseteq [r], |R| \geq 2} |R| \cdot b(R) \right) \\
&\geq \left( \sum_{R \subseteq [r], |R| \geq 3} 3 \cdot a(R) \right) + \left( \sum_{R \subseteq [r], |R| \geq 2} 2 \cdot b(R) \right) = 5.
\end{align*}

\noindent
This implies $a_i + b_i \geq \frac{5}{6}$ for some $i \in [6]$, as required.
\end{proof}

\begin{lemma}
Let $c_i, d_i \in [0,1]$ where $i \in [s]$, $k \in (0,1]$, and $[s] = I \sqcup O_1 \sqcup O_2$ where $|O_1| = |O_2| \geq 1$. If $c_i + d_i \leq k < \frac{2}{|I| + |O_1|}$ for all $i \in [s]$ and 

$$\sum_{i \in I \sqcup O_1} c_i \geq 1$$ 
$$\sum_{i \in I \sqcup O_2} c_i \geq 1.$$

\noindent
Let $e_1 = \frac{2+|I|k}{2|I|+|O_1|} - \frac{k}{2}$ and $e_2 = \frac{1+|I|\frac{k}{2}}{2|I| + |O_1|}$. Then we have 

$$\sum_{i \in [s]} c_id_i \leq |I| \cdot e_1(k-e_1) + 2|O_1| \cdot e_2(k - e_2).$$
\end{lemma}

\begin{proof}
We find an upper bound on the maximum of $\sum_{i \in [s]} c_id_i$ under the constraints above. The maximum exists because $\sum_{i \in [s]} c_id_i$ is continuous and the constrained space is compact with respect to $c_i,d_i$. We may assume $c_i + d_i = k$ for all $i \in [s]$, because if $c_j + d_j < k$ for some $j \in [s]$, then we may increase both $c_j,d_j$ so that $\sum_{i \in [s]} c_id_i$ increases. Thus we can just consider maximizing  

$$f(c_1,c_2,\ldots,c_{s}) = \sum_{i \in [s]} c_i(k-c_i).$$

We now show we can assume $\sum_{i \in I \sqcup O_1} c_i  = 1$ and $\sum_{i \in I \sqcup O_2} c_i  = 1$. We first note that the derivative of $g(x) = x(k-x)$ is $g'(x) = k-2x$, and so $g(x)$ is maximized at $x = \frac{k}{2}$. Furthermore, $g''(x) = -2$ so that if $x$ approaches $\frac{k}{2}$ we know $g(x)$ increases. If we had $\sum_{i \in I \sqcup O_1} c_i  > 1$, this would imply for some $j \in I \sqcup O_1$ we had $c_j > \frac{1}{|I| + |O_1|} > \frac{k}{2}$, where the last inequality comes from $k < \frac{2}{|I| + |O_1|}$. Thus we may decrease $c_j$ so that $f(c_1,c_2,\ldots,c_{s}) = \sum_{i \in [s]} c_i(k-c_i)$ attains a larger value while still maintaining the constraints. So let $\sum_{i \in I \sqcup O_1} c_i  = 1$ be the new constraint in place of $\sum_{i \in I \sqcup O_1} c_i  \geq 1$. \\

Now we show we may assume $\sum_{i \in I \sqcup O_2} c_i = 1$. So assume otherwise that $\sum_{i \in I \sqcup O_2} c_i > 1$. Note if for some $j \in O_2$ we had $c_j > \frac{k}{2}$, then we may decrease $c_j$ so that $f(c_1,c_2,\ldots,c_{s}) = \sum_{i \in [s]} c_i(k-c_i)$ attains a larger value while still maintaining the constraints. So assume for all $i \in O_2$, we have $c_i \leq \frac{k}{2}$. If for some $j \in O_2$ we had $c_j < \frac{k}{2}$, then we may increase $c_j$ so that $f(c_1,c_2,\ldots,c_{s}) = \sum_{i \in [s]} c_i(k-c_i)$ attains a larger value while still maintaining the constraints. So assume $c_i = \frac{k}{2}$ for all $i \in O_2$. Thus $\sum_{i \in O_2} c_i = |O_2|\frac{k}{2}$, so along with $\sum_{i \in I \sqcup O_2} c_i > 1$ we get $\sum_{i \in I} c_i > 1 - | O_2|\frac{k}{2}$. \\

Since $\sum_{i \in I \sqcup O_1} c_i = 1$ and $\sum_{i \in I} c_i > 1 - | O_2|\frac{k}{2}$, we must have 

$$\sum_{i \in O_1} c_i = 1-\sum_{i \in I} c_i < | O_2|\frac{k}{2}$$

\noindent
and thus for some $j \in O_1$ we have $c_j < \frac{|O_2|}{|O_1|} \frac{k}{2}  = \frac{k}{2}$. From $\sum_{i \in I} c_i > 1 - | O_2|\frac{k}{2}$ we also get for some $l \in I$ we have $c_l > \frac{1}{|I|} - \frac{|O_2|}{|I|}\frac{k}{2} > \frac{k}{2}$ where the last inequality can also be rearranged into $\frac{k}{2} < \frac{1}{|I|+|O_1|}$. \\

Note we can decrease $c_l$ by $\epsilon > 0$ and increase $c_j$ by $\epsilon$ so that $f(c_1,c_2,\ldots,c_{s}) = \sum_{i \in [s]} c_i(k-c_i)$ attains a larger value while still maintaining the constraints $\sum_{i \in I \sqcup O_1} c_i = 1$ and $\sum_{i \in I \sqcup O_2 } c_i \geq 1$. The constraint $\sum_{i \in I \sqcup O_1} c_i = 1$ is preserved because the $\epsilon$'s cancel out. The constraint $\sum_{i \in I \sqcup O_2 } c_i \geq 1$ can be preserved for small enough $\epsilon$ because by assumption $\sum_{i \in I \sqcup O_2 } c_i > 1$. Thus under the assumption $\sum_{i \in I \sqcup O_2 } k-c_i > 1$, we can always find a larger value of $f(c_1,c_2,\ldots,c_{s}) = \sum_{i \in [s]} c_i(k-c_i)$ while maintaining the desired constraints, so we may assume $\sum_{i \in I \sqcup O_2} c_i = 1$. \\

We claim under the constraints $\sum_{i \in I \sqcup O_1} c_i = 1$ and $\sum_{i \in I \sqcup O_2 } c_i = 1$, the function $f(c_1,c_2,\ldots,c_{s}) = \sum_{i \in [s]} c_i(k-c_i)$ has a maximum. Note we have

\begin{align*}
f(c_1,c_2,\ldots,c_{s}) &= \sum_{i \in [s]} c_i(k-c_i) = k \left( \sum_{i \in [s]} c_i \right) - \sum_{i \in [s]} c_i^2 \\
&= k \left( \sum_{i \in I \sqcup O_1} c_i + \sum_{i \in O_2} c_i \right) - \sum_{i \in [s]} c_i^2 \\ 
&= k \left( 1+ \sum_{i \in O_2} c_i \right) - \sum_{i \in [s]} c_i^2.
\end{align*}

\noindent
This shows for any $i \in I \sqcup O_1$, we have 

$$f(c_1,c_2, \ldots, c_i, \ldots, c_s) = f(c_1,c_2, \ldots, |c_i|, \ldots, c_s)$$

\noindent
and analogously 

\begin{align*}
f(c_1,c_2,\ldots,c_{s}) &= \sum_{i \in [s]} c_i(k-c_i) = k \left( \sum_{i \in [s]} c_i \right) - \sum_{i \in [s]} c_i^2 \\
&= k \left( \sum_{i \in O_1} c_i + \sum_{i \in I \sqcup O_2} c_i \right) - \sum_{i \in [s]} c_i^2 \\
&= k \left(1 +  \sum_{i \in O_1} c_i \right) - \sum_{i \in [s]} c_i^2
\end{align*}

\noindent 
which shows for any $i \in I \sqcup O_2$, we have 

$$f(c_1,c_2, \ldots, c_i, \ldots, c_s) = f(c_1,c_2, \ldots, |c_i|, \ldots, c_s).$$ 

\noindent
Thus $f(c_1,c_2,\ldots,c_s) = f(|c_1|,|c_2|,\ldots,|c_s|)$, so we may assume $c_i \geq 0$ for all $i \in [s]$. Note along with the assumption $c_i \geq 0$ for all $i \in [s]$, the constraints $\sum_{i \in I \sqcup O_1} c_i = 1$ and $\sum_{i \in I \sqcup O_2 } c_i = 1$ give rise to a compact space, and since $f$ is continuous we know $f$ attains a maximum in this compact space defined by the assumption $c_i \geq 0$ for all $i \in [s]$ with the constraints $\sum_{i \in I \sqcup O_1} c_i = 1$ and $\sum_{i \in I \sqcup O_2 } c_i = 1$, and thus $f$ attains a maximum with the constraints $\sum_{i \in I \sqcup O_1} c_i = 1$ and $\sum_{i \in I \sqcup O_2 } c_i = 1$ only. \\

We use Lagrange multipliers to maximize $f(c_1,c_2,\ldots,c_{s}) = \sum_{i \in [s]} c_i(k-c_i)$ under the constraints $\sum_{i \in I \sqcup O_1} c_i = 1$ and $\sum_{i \in I \sqcup O_2 } c_i = 1$. We have $\frac{\partial f}{\partial c_i} = k - 2c_i$ for all $i \in [s]$. Define $g_1(c_1,c_2,\ldots,c_s) = \sum_{i \in I \sqcup O_1} c_i$ and $g_2(c_1,c_2,\ldots,c_s) = \sum_{i \in I \sqcup O_2 } c_i$. So $\frac{\partial g_1}{\partial c_i} = 1$ if $i \in I \cup O_1$ and $0$ otherwise, and $\frac{\partial g_2}{\partial c_i} = 1$ if $i \in I \cup O_2$ and $0$ otherwise. Thus for some $\lambda_1,\lambda_2 \in \mathbb{R}$ we have 
$k-2c_i = \lambda_1+\lambda_2$ for all $i \in I$, and $k-2c_i = \lambda_1$ for all $i \in O_1$, and $k-2c_i = \lambda_2$ for all $i \in O_2$. This gives us for all $i,j \in I$ we have $c_i = c_j$, for all $i,j \in O_1$ we have $c_i = c_j$, and for all $i,j \in O_2$ we have $c_i = c_j$. \\

Without loss of generality, let $c_1 \in O_1, c_2 \in O_2, c_3 \in I$. We get 

$$\sum_{i \in I \sqcup O_1} c_i = \sum_{i \in I} c_i + \sum_{i \in O_1} c_i = \sum_{i \in I} c_3 + \sum_{i \in O_1} c_1 = |I|c_3 + |O_1|c_1 = 1.$$

\noindent
Likewise

$$\sum_{i \in I \sqcup O_2} c_i = \sum_{i \in I}c_i + \sum_{i \in O_2}c_i = |I|c_3 + |O_2|c_2 = 1.$$

\noindent
Thus we get $|I|c_3 + |O_1|c_1 =|I|c_3 + |O_2|c_2$, and so $|O_1|c_1 = |O_2|c_2$ so that $c_1 = c_2$. Recall also that $k-2c_3 = \lambda_1 + \lambda_2 = k-2c_1+k-2c_2  = 2k-4c_1$ and thus $c_3 = 2c_1-\frac{k}{2}$. Recall $|I|c_3 + |O_1|c_1 = 1$ and so $c_1 = \frac{1}{|O_1|}(1-|I|c_3)$, which we substitute into $c_3 = 2c_1-\frac{k}{2}$ to get 

$$c_3 = \frac{2}{|O_1|}(1-|I|c_3)-\frac{k}{2}$$

\noindent
which we re-arrange to get 

$$c_3(\frac{2|I|+|O_1|}{|O_1|}) = \frac{2}{|O_1|} - \frac{k}{2}$$

\noindent
and thus 

\begin{align*}
c_3 &= \frac{|O_1|}{2|I|+|O_1|} (\frac{2}{|O_1|} - \frac{k}{2}) = \frac{2}{2|I|+|O_1|} - \frac{|O_1|}{2|I|+|O_1|}\frac{k}{2} \\
&=  \frac{2}{2|I|+|O_1|} - \frac{k}{2} + \frac{2|I|}{2|I|+|O_1|}\frac{k}{2} = \frac{2+|I|k}{2|I|+|O_1|} - \frac{k}{2} = e_1.
\end{align*}

We have $c_1 = \frac{1}{2}(c_3 + \frac{k}{2})$ so that

$$c_1 = \frac{1}{2}( \frac{2+|I|k}{2|I|+|O_1|} - \frac{k}{2} + \frac{k}{2}) =  \frac{1+|I|\frac{k}{2}}{2|I|+|O_1|} = e_2.$$

Thus

\begin{align*}
f(c_1,c_2,\ldots,c_s) &= |I|e_1(k-e_1) + |O_1|e_2(k-e_2) + |O_2|e_2(k-e_2) \\
&= |I| \cdot e_1(k-e_1) + 2|O_1| \cdot e_2(k-e_2) 
\end{align*}

\noindent
as required. 
\end{proof}

\begin{proposition}
We have $g^*(7) = \frac{7}{9}$.
\end{proposition}

\begin{proof}
To see that $g^*(7) \leq \frac{7}{9}$, consider \\

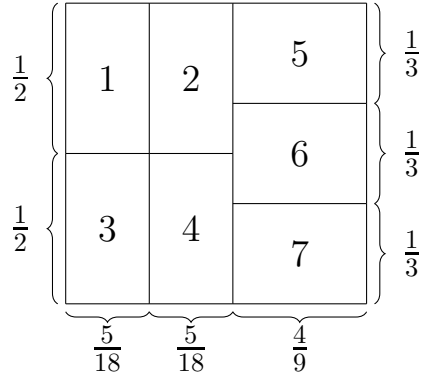
\begin{center}
\begin{tikzpicture}[declare function={a=4;}]

\draw (0,0) coordinate (A) ;
\draw (5*a/18,0) coordinate (B) ;
\draw (5*a/9,0) coordinate (C) ;
\draw (0,a/2) coordinate (D) ;
\draw (5*a/18,a/2) coordinate (E) ;
\draw (5*a/9,a/2) coordinate (F) ;
\draw (0,a) coordinate (G) ;
\draw (5*a/18,a) coordinate (H) ;
\draw (5*a/9,a) coordinate (I) ;
\draw (5*a/9,a/3) coordinate (J) ;
\draw (5*a/9,2*a/3) coordinate (K) ;
\draw (a,0) coordinate (L) ;
\draw (a,a/3) coordinate (M) ;
\draw (a,2*a/3) coordinate (N) ;
\draw (a,a) coordinate (O) ;

\draw (A) -- (B) ;
\draw (B) -- (C) ;
\draw (D) -- (E) ;
\draw (E) -- (F) ;
\draw (G) -- (H) ;
\draw (H) -- (I) ;
\draw (A) -- (D) ;
\draw (D) -- (G) ;
\draw (B) -- (E) ;
\draw (E) -- (H) ;
\draw (C) -- (F) ;
\draw (F) -- (I) ;
\draw (C) -- (L) ;
\draw (J) -- (M) ;
\draw (K) -- (N) ;
\draw (I) -- (O) ;
\draw (L) -- (M) ;
\draw (M) -- (N) ;
\draw (N) -- (O) ;

\draw (5*a/36,3*a/4) node{\Large 1} ;
\draw (15*a/36,3*a/4) node{\Large 2} ;
\draw (5*a/36,a/4) node{\Large 3} ;
\draw (15*a/36,a/4) node{\Large 4} ;
\draw (7*a/9,5*a/6) node{\Large 5} ;
\draw (7*a/9,a/2) node{\Large 6} ;
\draw (7*a/9,a/6) node{\Large 7} ;

\draw[decorate,decoration={brace,raise=3pt,amplitude=4pt}] (A) -- (D) node[midway,xshift=-0.6cm]{\Large $\frac{1}{2}$} ;
\draw[decorate,decoration={brace,raise=3pt,amplitude=4pt}] (D) -- (G) node[midway,xshift=-0.6cm]{\Large $\frac{1}{2}$} ;
\draw[decorate,decoration={brace,raise=3pt,amplitude=4pt,mirror}] (A) -- (B) node[midway,yshift=-0.6cm]{\Large $\frac{5}{18}$} ;
\draw[decorate,decoration={brace,raise=3pt,amplitude=4pt,mirror}] (B) -- (C) node[midway,yshift=-0.6cm]{\Large $\frac{5}{18}$} ;
\draw[decorate,decoration={brace,raise=3pt,amplitude=4pt,mirror}] (C) -- (L)  node[midway,yshift=-0.6cm]{\Large $\frac{4}{9}$} ;
\draw[decorate,decoration={brace,raise=3pt,amplitude=4pt,mirror}] (L) -- (M) node[midway,xshift=0.6cm]{\Large $\frac{1}{3}$} ;
\draw[decorate,decoration={brace,raise=3pt,amplitude=4pt,mirror}] (M) -- (N) node[midway,xshift=0.6cm]{\Large $\frac{1}{3}$} ;
\draw[decorate,decoration={brace,raise=3pt,amplitude=4pt,mirror}] (N) -- (O) node[midway,xshift=0.6cm] {\Large $\frac{1}{3}$} ;
\end{tikzpicture}
\end{center}

Now we show $g^*(7) \geq \frac{7}{9}$. Consider any solution $a,b : \mathcal{P}([7]) \rightarrow [0,1]$ satisfying constraints (1), (2), (3), and (5). We may assume $a_i + b_i \leq \frac{7}{9}$ for all $i \in [7]$. By Lemma 5 we have that if $|R| \leq 1$ then $a(R) = b(R) = 0$, since $a_i + b_i \leq \frac{7}{9} < \frac{1}{1}$ for all $i \in [7]$. \\

We claim there do not exist $R_1,R_2 \subseteq [7]$ such that $a(R_1), b(R_2) > 0$ and $|R_1|=|R_2| = 2$. Assume otherwise for the sake of contradiction. By constraint (5), we have $R_1 \cap R_2 \neq \{\}$, so we may without loss of generality assume $1 \in R_1 \cap R_2$. If without loss of generality $R_1 = R_2 = \{1,2\}$, then we have $a(\{1,2\}),b(\{1,2\}) > 0$ which implies $b_1+b_2 \geq 1$ and $a_1+a_2 \geq 1$ respectively, so that $(a_1+b_1) + (a_2+b_2) \geq 2$ and thus for some $i \in [2]$ we have $a_i+b_i \geq 1$, contradicting $a_i+b_i \leq \frac{7}{9}$. Now without loss of generality assume $R_1 = \{1,2\}$ and $R_2 = \{1,3\}$ so that $a(\{1,2\}) > 0$ and $b(\{1,3\}) > 0$ which implies $b_1 + b_2 \geq 1$ and $a_1 + a_3 \geq 1$ respectively. Apply Lemma 6 with $I = \{1\}, O_1 = \{2\}, O_2 = \{3\}$ and $k = \frac{7}{9}$. We just need to check $\frac{7}{9} < \frac{2}{1+1}$ which is true. Thus Lemma 6 says that $a_1b_1 + a_2b_2 + a_3b_3 \leq 0.36$. Thus we have $ 3 \cdot (\frac{1}{2} \cdot \frac{7}{9})^2 \geq a_4b_4 + a_5b_5 + a_6b_6 \geq 1-0.36 = 0.64$, a contradiction. \\

So without loss of generality assume if $b(R) > 0$, then $|R| \geq 3$, so $\sum_{i \in [7]} b_i = \sum_{R \subseteq [7]} b(R) \geq 3$, and thus since $\sum_{i \in [7]} a_i + b_i \leq 7 \cdot \frac{7}{9} = \frac{49}{9}$, we must have that 

$$\sum_{R \subseteq [7]} |R| \cdot a(R) = \sum_{i \in [7]} a_i \leq \frac{49}{9} - 3 = \frac{22}{9} = 2+ \frac{4}{9}.$$

\noindent
Thus

\begin{align*}
2+ \frac{4}{9} \geq \sum_{R \subseteq [7]} |R| \cdot a(R) &= \left( \sum_{R \subseteq [7], |R| = 2} |R| \cdot a(R) \right) + \left( \sum_{R \subseteq [7], |R| \geq 3} |R| \cdot a(R) \right) \\
&\geq 2\left( \sum_{R \subseteq [7], |R| = 2} a(R) \right) + 3\left( \sum_{R \subseteq [7], |R| \geq 3} a(R) \right) \\
&= 2 + \left( \sum_{R \subseteq [7], |R| \geq 3} a(R) \right)
\end{align*}

\noindent
so that $\sum_{R \subseteq [7], |R| \geq 3} a(R) \leq \frac{4}{9}$ and so $\sum_{R \subseteq [7], |R| = 2} a(R) \geq 1- \frac{4}{9} = \frac{5}{9}$. So without loss of generality assume $a(\{1,2\}) > 0$ so that $b_1 + b_2 \geq 1$. Since $(a_1 + b_1) + (a_2+b_2) \leq \frac{14}{9}$, we have $a_1 + a_2 \leq \frac{5}{9}$. Note since $a(\{1,2\}) > 0$, we have 

$$\sum_{R \in \{1\}^{\uparrow} \cup \{2\}^{\uparrow}} a(R) \leq a_1 + a_2 - a(\{1,2\}) < \frac{5}{9}.$$

\noindent
Recall $\sum_{R \subseteq [7], |R| = 2} a(R) \geq \frac{5}{9}$, and thus we may without loss of generality assume that $a(\{3,4\}) > 0$ as well. Thus we get that $b_3 + b_4 \geq 1$. \\

\textbf{Claim 1}: There can't exist $i,j \in \{5,6,7\}, i \neq j$ such that $i \in R_1$, $j \in R_2$, $a(R_1),a(R_2) > 0$, and $|R_1| = |R_2| = 2$. Without loss of generality, assume $i = 5, j = 6$. \\

\textbf{Case 1}: $R_1 = R_2 = \{5,6\}$. $a(\{5,6\}) > 0$ would imply $b_5 + b_6 \geq 1$. Thus we have $\sum_{i \in [6]} b_i \geq 3 > 2.93$. This contradicts Corollary 1. \\

\textbf{Case 2}: $R_1 = \{1,5\}, R_2 = \{1,6\}$. Let $P \subseteq \mathcal{P}([7])$ be defined as 

$$P = \{R \subseteq [7] : 2,5,6 \in R, \ 3 \in R \ \vee \ 4 \in R \}.$$

From $a(\{1,2\}) > 0, a(\{1,5\}) > 0, a(\{1,6\}) > 0$, we have that if $1 \not \in R$ and $b(R) > 0$, we must have $2,5,6 \in R$ by constraint (5). Also since $a(\{3,4\}) > 0$, we have $b(R) > 0$ implies $3 \in R \ \vee \ 4  \in R$ by constraint (5) as well. Thus if $1 \not \in R$ and $b(R) > 0$, we have $R \in P$. Also note that by Corollary 1 we have $b_1 \leq 0.6311$. This implies that 

$$\sum_{R \subseteq [7], \ |R| \geq 4} b(R) \geq \sum_{R \in P} b(R) \geq \sum_{R \subseteq [7], \ 1 \not\in R} b(R) \geq 1 - 0.6311 = 0.3689.$$

\noindent
Recall that $b(R) > 0$ implies $|R| \geq 3$, thus we get 

\begin{align*}
\sum_{i \in [7]} b_i = \sum_{R \subseteq [7]} |R| \cdot b(R) &= \left( \sum_{R \subseteq [7], |R| = 3} |R| \cdot b(R) \right) + \left( \sum_{R \subseteq [7], |R| \geq 4} |R| \cdot b(R) \right) \\
&\geq 3 \left( \sum_{R \subseteq [7], |R| = 3} b(R) \right) + 4\left( \sum_{R \subseteq [7], |R| \geq 4} b(R) \right) \\
&= 3+\left( \sum_{R \subseteq [7], |R| \geq 4} b(R) \right) \\
&\geq 3 + 0.3689  > 3.363 
\end{align*}

\noindent
which contradicts Corollary 1. \\

\textbf{Case 3}: $R_1 = \{1,5\}, R_2 = \{2,6\}$. This implies that if $b(R) > 0$, and $1 \not \in R$ we must have $5 \in R$ by constraint (5). Likewise, $b(R) > 0, \ 2 \not \in R$ implies $6 \in R$. This implies that $b_5 + b_6 \geq 1 - b'(\{1,2\}^{\uparrow})$. Furthermore, recall that $a(\{1,2\}) > 0$ so that $b(R) > 0, \ 2 \not\in R$ implies $1 \in R$, and thus $ b'(\{1\}^{\uparrow} \setminus \{2\}^{\uparrow}) + b'(\{2\}^{\uparrow}) = 1$. Thus

$$b_1+b_2+b_5+b_6 \geq  b'(\{1\}^{\uparrow}) + b'(\{2\}^{\uparrow}) + 1 - b'(\{1,2\}^{\uparrow}) $$

$$=  b'(\{1\}^{\uparrow} \setminus \{1,2\}^{\uparrow}) + b'(\{2\}^{\uparrow}) + 1  = b'(\{1\}^{\uparrow} \setminus \{2\}^{\uparrow}) + b'(\{2\}^{\uparrow}) + 1 = 2$$

\noindent
where the second to last equality comes from $\{1\}^{\uparrow} \setminus \{2\}^{\uparrow} = \{1\}^{\uparrow} \setminus \{1,2\}^{\uparrow}$. Recall we also had $b_3 + b_4 \geq 1$, and thus $\sum_{i \in [6]} b_i \geq 3$, contradicting Corollary 1. \\

\textbf{Case 4}: $R_1 = \{1,5\}, R_2 = \{3,6\}$. Note $a(\{1,5\}) > 0$ implies $b_1 + b_5 \geq 1$. Recall also that $b_1 + b_2 \geq 1$. We apply Lemma 7 with $I = \{1\}, O_1 = \{2\}, O_2 = \{5\}$ with $k = \frac{7}{9} < \frac{2}{1+1}$. Thus we get $a_1b_1 + a_2b_2 + a_5b_5 \leq  \frac{409}{972}$. One can also show $a_3b_3 + a_4b_4 + a_6b_6 \leq \frac{409}{972}$, because it satisfies similar constraints from $b(\{3,6\}) > 0$ so that $b_3 + b_6 \geq 1$ and recall $b_3 + b_4 \geq 1$. Thus $\sum_{i \in [6]} a_ib_i \leq \frac{818}{972}$. This implies we must have $a_7b_7 \geq 1-\frac{818}{972} = \frac{154}{972} > (\frac{1}{2} \cdot \frac{7}{9})^2 \geq a_7b_7$, a contradiction. \\

Other cases are analogous to the four cases described above, so the claim is proven. We may without loss of generality assume that if $a(R) > 0$ and $|R| = 2$, then $R \subseteq [5]$, so that if $a(R) > 0$  and $6 \in R \ \vee \  7 \in R$ then $|R| \geq 3$. \\

\textbf{Claim 2:} there can't exist $i \in \{5,6,7\}$ such that $i \in R$, $|R| = 2$, and $a(R) > 0$. Assume otherwise, and without loss of generality let $i = 5$, and $R = \{1,5\}$. Thus $b_1 + b_5 \geq 1$ since $a(\{1,5\}) > 0$. By Corollary 1, we have $b_1 \leq 0.64$ so that $b_5 \geq 0.36$. Recall $b_1 + b_2 \geq 1, \ b_3 + b_4 \geq 1$, so that we get $\sum_{i \in [5]} b_i \geq 2.36$. Thus $\sum_{i \in [5]} a_i \leq 5 \cdot \frac{7}{9} - 2.36$. \\

 By Claim 1, we also have $a(R) > 0$ and $|R| = 2$ implies $R \subseteq [5]$, so that 

$$\sum_{R \subseteq [7], \ |R| = 2} a(R) = \sum_{R \subseteq [5], \ |R| = 2} a(R).$$

\noindent
We note that if $R \subseteq [7], \ |R| \geq 3$, then for some $i \in [5]$ we have $i \in R$. Recall $\sum_{R \subseteq [7], \ |R| = 2} a(R) \geq \frac{5}{9}$. This means 

\begin{align*}
\sum_{i \in [5]} a_i &\geq 2 \left( \sum_{R \subseteq [5], \ |R| = 2} a(R) \right) + \left( \sum_{R \subseteq [7], \ |R| \geq 3} a(R) \right) \\
&= \left( \sum_{R \subseteq [5], \ |R| = 2} a(R) \right) + \left( \sum_{R \subseteq [7], \ |R| = 2} a(R) + \sum_{R \subseteq [7], \ |R| \geq 3} a(R) \right) \\
&= \left( \sum_{R \subseteq [5], \ |R| = 2} a(R) \right) + 1 \geq \frac{5}{9} + 1 =\frac{14}{9}.
\end{align*}

\noindent
But then we have $5 \cdot \frac{7}{9} - 2.36 \geq \sum_{i \in [5]} a_i \geq \frac{14}{9}$, a contradiction. Thus Claim 2 is proven. In particular, Claim 2 says that if $R \subseteq [7], \ a(R) > 0$ and $|R| = 2$, then $R \subseteq [4]$. So $\sum_{R \subseteq [7], \ |R| = 2} a(R) = \sum_{R \subseteq [4], \ |R| = 2} a(R) \geq \frac{5}{9}$. But note that since $b_1 + b_2 \geq 1, \ b_3 + b_4 \geq 1$, we have that $\sum_{i \in [4]} a_i \leq 4 \cdot \frac{7}{9} - 2 = \frac{10}{9}$. Note

$$\frac{10}{9} \geq \sum_{i\in [4]} a_i \geq \sum_{R \subseteq [4], \ |R| = 2} 2 \cdot a(R) \geq \frac{10}{9}.$$

In particular, this implies $\sum_{i\in [4]} a_i = \frac{10}{9}$. Along with the equations $b_1 + b_2 \geq 1, \ b_3 + b_4 \geq 1$, we get $\sum_{i \in [4]} a_i + b_i \geq \frac{10}{9} + 2 = \frac{28}{9}$, so that for some $i \in [4]$ we know $a_i + b_i \geq \frac{1}{4} \cdot \frac{28}{9} = \frac{7}{9}$, as required. 
\end{proof}

\begin{proposition}
We have $g^*(8) = \frac{13}{18}$.
\end{proposition}

\begin{proof}
To see that $g^*(8) \leq \frac{13}{18}$, consider \\

\begin{center}
\begin{tikzpicture}[declare function={a=4;}]
\draw (0,0) coordinate (A) ;
\draw (7*a/18,0) coordinate (B) ;
\draw (7*a/9,0) coordinate (C) ;
\draw (0,a/3) coordinate (D) ;
\draw (7*a/18,a/3) coordinate (E) ;
\draw (7*a/9,a/3) coordinate (F) ;
\draw (0,2*a/3) coordinate (G) ;
\draw (7*a/18,2*a/3) coordinate (H) ;
\draw (7*a/9,2*a/3) coordinate (I) ;
\draw (0,a) coordinate (J) ;
\draw (7*a/18,a) coordinate (K) ;
\draw (7*a/9,a) coordinate (L) ;
\draw (7*a/9,a/2) coordinate (M) ;
\draw (a,0) coordinate (O) ;
\draw (a,a/2) coordinate (P) ;
\draw (a,a) coordinate (Q) ;

\draw (A) -- (B) ;
\draw (B) -- (C) ;
\draw (D) -- (E) ;
\draw (E) -- (F) ;
\draw (G) -- (H) ;
\draw (H) -- (I) ;
\draw (J) -- (K) ;
\draw (K) -- (L) ;
\draw (A) -- (D) ;
\draw (B) -- (E) ;
\draw (C) -- (F) ;
\draw (D) -- (G) ;
\draw (E) -- (H) ;
\draw (F) -- (I) ;
\draw (G) -- (J) ;
\draw (H) -- (K) ;
\draw (I) -- (L) ;
\draw (C) -- (O) ;
\draw (M) -- (P) ;
\draw (L) -- (Q) ;
\draw (P) -- (Q) ;
\draw (O) -- (P) ;

\draw (7*a/39,5*a/6) node{\Large 1} ;
\draw (21*a/36,5*a/6) node{\Large 2} ;
\draw (7*a/39,a/2) node{\Large 3} ;
\draw (21*a/36,a/2) node{\Large 4} ;
\draw (7*a/39,a/6) node{\Large 5} ;
\draw (21*a/36,a/6) node{\Large 6} ;
\draw (8*a/9,3*a/4) node{\Large 7} ;
\draw (8*a/9,a/4) node{\Large 8} ;

\draw[decorate,decoration={brace,raise=3pt,amplitude=4pt}] (A) -- (D) node[midway,xshift=-0.6cm]{\Large $\frac{1}{3}$} ;
\draw[decorate,decoration={brace,raise=3pt,amplitude=4pt}] (D) -- (G) node[midway,xshift=-0.6cm]{\Large $\frac{1}{3}$} ;
\draw[decorate,decoration={brace,raise=3pt,amplitude=4pt}] (G) -- (J) node[midway,xshift=-0.6cm]{\Large $\frac{1}{3}$} ;
\draw[decorate,decoration={brace,raise=3pt,amplitude=4pt,mirror}] (A) -- (B) node[midway,yshift=-0.6cm]{\Large $\frac{7}{18}$} ;
\draw[decorate,decoration={brace,raise=3pt,amplitude=4pt,mirror}] (B) -- (C) node[midway,yshift=-0.6cm]{\Large $\frac{7}{18}$} ;
\draw[decorate,decoration={brace,raise=3pt,amplitude=4pt,mirror}] (C) -- (O)  node[midway,yshift=-0.6cm]{\Large $\frac{2}{9}$} ;
\draw[decorate,decoration={brace,raise=3pt,amplitude=4pt,mirror}] (O) -- (P) node[midway,xshift=0.6cm]{\Large $\frac{1}{2}$} ;
\draw[decorate,decoration={brace,raise=3pt,amplitude=4pt,mirror}] (P) -- (Q) node[midway,xshift=0.6cm]{\Large $\frac{1}{2}$} ;

\end{tikzpicture}
\end{center}

Now we show $g^*(8) \geq \frac{13}{18}$. Consider any solution $a,b : \mathcal{P}([8]) \rightarrow [0,1]$ satisfying constraints (1), (2), (3), and (5). We may assume $a_i + b_i \leq \frac{13}{18}$ for all $i \in [8]$. By Lemma 5 we have for $|R| \leq 1$ that $a(R) = b(R) = 0$, as $a_i+b_i \leq \frac{13}{18} < \frac{1}{1}$ for all $i \in [8]$.  \\

We first show there is no $R_1,R_2 \subseteq [8]$ such that $|R_1| = |R_2| = 2$ and $a(R_1) > 0, b(R_2) > 0$. Assume for the sake of contradiction that this is not the case. If without loss of generality $R_1 = R_2 = \{1,2\}$, then we have from $a(\{1,2\}) > 0$ that $b_1 + b_2 \geq 1$, and likewise $b(\{1,2\}) > 0$ implies $a_1 + b_1 \geq 1$. This means $(a_1+b_1) + (a_2+b_2) \geq 2$ so that for some $i \in \{1,2\}$ we have $a_i+b_i \geq 1$, contradicting $a_i + b_i \leq \frac{13}{18}$ for all $i \in [8]$. Now without loss of generality assume $R_1 = \{1,2\}, R_2 = \{1,3\}$. From $a(\{1,2\}) > 0$ we have that $b_1 + b_2 \geq 1$, and from $b(\{1,3\})$ we have $a_1 + a_3 \geq 1$. So we have $a_1 + b_1 + b_2 + a_3 \geq 2$. Since $a_1 + b_1 \leq \frac{13}{18}$, we have that $b_2 + a_3 \geq 2 - \frac{13}{18} = \frac{23}{18} > 1.02$, which contradicts Corollary 1. \\

Thus we may without loss of generality assume that if $b(R) > 0$, then $|R| \geq 3$. Note

$$\sum_{R \subseteq [8]} |R|(a(R) + b(R)) = \sum_{i \in [8]} a_i + b_i \leq 8 \cdot \frac{13}{18} $$

\noindent
so that

\begin{align*}
8 \cdot \frac{13}{18} &\geq \sum_{R \subseteq [8]} (a(R) + b(R))|R| = \sum_{R \subseteq [8]} |R| \cdot a(R) \ + \sum_{R \subseteq [8], |R| \geq 3} |R| \cdot b(R) \\
&\geq \sum_{R \subseteq [8]} |R| \cdot a(R) + \sum_{R \subseteq [8], |R| \geq 3} 3 \cdot b(R)  = \sum_{R \subseteq [8]} |R| \cdot a(R) + 3
\end{align*}

\noindent
and thus $\sum_{R \subseteq [8]} |R| \cdot a(R) \leq 8 \cdot \frac{13}{18} - 3 = \frac{50}{18} = 2 + \frac{7}{9}$. In particular, this implies that 

\begin{align*}
 2 + \frac{7}{9} \geq \sum_{R \subseteq [8]} |R| \cdot a(R) &= \left( \sum_{R \subseteq [8], |R| = 2} |R| \cdot a(R) \right) + \left( \sum_{R \subseteq [8], |R| \geq 3} |R| \cdot a(R) \right) \\
 &\geq 2\left( \sum_{R \subseteq [8], |R| = 2} a(R) \right) + 3\left( \sum_{R \subseteq [8], |R| \geq 3} a(R) \right) \\
&= 2 + \left( \sum_{R \subseteq [8], |R| \geq 3} a(R) \right)
\end{align*}

\noindent
so that $\sum_{R \subseteq [8], |R| \geq 3} a(R) \leq \frac{7}{9}$ and thus $\sum_{R \subseteq [8], |R| = 2} a(R) \cdot|R| \geq \frac{2}{9}$. Without loss of generality, let $a(\{1,2\}) > 0$, so that $b_1 + b_2 \geq 1$. \\

We now show we cannot have $i \in [8] \setminus \{1,2\}$ such that $i \in R, \ |R| = 2$ and $a(R) > 0$. Assume otherwise for the sake of contradiction, and without loss of generality let $i = 3$. The first case is if without loss of generality $R = \{3,4\}$. From $a(\{3,4\}) > 0$ we have $b_3 + b_4 \geq 1$. Thus we get $b_1 + b_2 + b_3 + b_4 \geq 2 > 1,861$ which contradicts Corollary 1. The second case is if without loss of generality $R = \{1,3\}$, so that $a(\{1,3\}) > 0$ and thus $b_1 + b_3 \geq 1$. Consider the maximum value $a_1b_1 + a_2b_2 + a_3b_3$ can take on given the constraints $b_1 + b_2 \geq 1$, $b_1 + b_3 \geq 1$, and $a_i+b_i \leq \frac{13}{18}$ for $i \in [3]$. We apply Lemma 7 with $I = \{1\}, O_1 = \{2\}, O_2 = \{3\}$ with $k = \frac{13}{18} < \frac{2}{1+1}$ to get that $a_1b_1 + a_2b_2 + a_3b_3 \leq \frac{1321}{3888}$. This implies that $5 (\frac{1}{2} \cdot \frac{13}{18})^2 \geq \sum_{i \in [8] \setminus [3]} a_ib_i \geq 1 - \frac{1321}{3888}$, a contradiction. \\

Thus $a(R) > 0$ and $|R| = 2$ implies $R = \{1,2\}$. Since 

$$\sum_{R \subseteq [8], |R| = 2} |R| \cdot a(R) = a(\{1,2\}) \geq \frac{2}{9}$$

\noindent
we have $a_1, a_2 \geq \frac{2}{9}$ so that $a_1 + a_2 \geq \frac{4}{9}$. Recall also that $b_1 + b_2 \geq 1$, so that together we have $(a_1+b_1) + (a_2+b_2) \geq \frac{13}{9}$. This implies for some $i \in [2]$ we have $a_i+b_i \geq \frac{13}{18}$, as required. 
\end{proof}

\begin{center}
\section{$g^*(t^2-1) = \frac{1}{t}-\frac{1}{t^2}+\frac{1}{t-1}$}
\end{center}

Note when $t = 3$, we have $g^*(3^2-1) = g^*(8) = \frac{1}{3} - \frac{1}{3^2} + \frac{1}{3-1} = \frac{13}{18}$. The proof for $g^*(8) = \frac{13}{18}$ can be generalized. We first describe a construction showing $g^*(t^2-1) \leq \frac{1}{t}-\frac{1}{t^2}+\frac{1}{t-1}$ for all $t \geq 3$ (the equality holds true for $t = 2$ as well, but the proof for the equality fails in this case). Then we give an example for when $t = 5$. \\

For general $t \geq 3$, split the square into two rectangles each with vertical side length $1$. The rectangle on the right will have horizontal side length $\frac{1}{t} - \frac{1}{t^2}$ and the rectangle on the left will have horizontal side length $1-\frac{1}{t}+\frac{1}{t^2}$. Partition the right rectangle into $t-1$ smaller rectangles each with vertical side length $\frac{1}{t-1}$ and horizontal side length $\frac{1}{t} - \frac{1}{t^2}$. Partition the left rectangle into $(t-1)t$ smaller rectangles each with vertical side length $\frac{1}{t}$ and horizontal side length $\frac{1}{t-1}(1-\frac{1}{t}+\frac{1}{t^2}) = \frac{1}{t-1} - \frac{1}{t^2}$. \\

To see that $g^*(5^2-1) = g^*(24) \leq \frac{1}{5}-\frac{1}{5^2}+\frac{1}{5-1} = \frac{41}{100}$, note that for $t = 5$ we have $\frac{1}{t-1} = \frac{1}{4}$, $\frac{1}{t} = \frac{1}{5}$, $\frac{1}{t} - \frac{1}{t^2} = \frac{4}{25}$, and $\frac{1}{t-1}-\frac{1}{t^2} = \frac{21}{100}$ and consider \\

\begin{center}
\begin{tikzpicture}[declare function={a=5;}]

\draw (0,0) coordinate (1) ;
\draw (21*a/100,0) coordinate (2) ;
\draw (21*a/50,0) coordinate (3) ;
\draw (63*a/100,0) coordinate (4) ;
\draw (21*a/25,0) coordinate (5) ;
\draw (0,a/5) coordinate (6) ;
\draw (21*a/100,a/5) coordinate (7) ;
\draw (21*a/50,a/5) coordinate (8) ;
\draw (63*a/100,a/5) coordinate (9) ;
\draw (21*a/25,a/5) coordinate (10) ;
\draw (0,2*a/5) coordinate (11) ;
\draw (21*a/100,2*a/5) coordinate (12) ;
\draw (21*a/50,2*a/5) coordinate (13) ;
\draw (63*a/100,2*a/5) coordinate (14) ;
\draw (21*a/25,2*a/5) coordinate (15) ;
\draw (0,3*a/5) coordinate (16) ;
\draw (21*a/100,3*a/5) coordinate (17) ;
\draw (21*a/50,3*a/5) coordinate (18) ;
\draw (63*a/100,3*a/5) coordinate (19) ;
\draw (21*a/25,3*a/5) coordinate (20) ;
\draw (0,4*a/5) coordinate (21) ;
\draw (21*a/100,4*a/5) coordinate (22) ;
\draw (21*a/50,4*a/5) coordinate (23) ;
\draw (63*a/100,4*a/5) coordinate (24) ;
\draw (21*a/25,4*a/5) coordinate (25) ;
\draw (0,a) coordinate (26) ;
\draw (21*a/100,a) coordinate (27) ;
\draw (21*a/50,a) coordinate (28) ;
\draw (63*a/100,a) coordinate (29) ;
\draw (21*a/25,a) coordinate (30) ;
\draw (21*a/25,a/4) coordinate (31) ;
\draw (21*a/25,a/2) coordinate (32) ;
\draw (21*a/25,3*a/4) coordinate (33) ;
\draw (a,0) coordinate (34) ;
\draw (a,a/4) coordinate (35) ;
\draw (a,a/2) coordinate (36) ;
\draw (a,3*a/4) coordinate (37) ;
\draw (a,a) coordinate (38) ;

\draw (1) -- (2) ;
\draw (2) -- (3) ;
\draw (3) -- (4) ;
\draw (4) -- (5) ;
\draw (6) -- (7) ;
\draw (7) -- (8) ;
\draw (8) -- (9) ;
\draw (9) -- (10) ;
\draw (11) -- (12) ;
\draw (12) -- (13) ;
\draw (13) -- (14) ;
\draw (14) -- (15) ;
\draw (16) -- (17) ;
\draw (17) -- (18) ;
\draw (18) -- (19) ;
\draw (19) -- (20) ;
\draw (21) -- (22) ;
\draw (22) -- (23) ;
\draw (23) -- (24) ;
\draw (24) -- (25) ;
\draw (26) -- (27) ;
\draw (27) -- (28) ;
\draw (28) -- (29) ;
\draw (29) -- (30) ;
\draw (1) -- (6) ;
\draw (6) -- (11) ;
\draw (11) -- (16) ;
\draw (16) -- (21) ;
\draw (21) -- (26) ;
\draw (2) -- (7) ;
\draw (7) -- (12) ;
\draw (12) -- (17) ;
\draw (17) -- (22) ;
\draw (22) -- (27) ;
\draw (3) -- (8) ;
\draw (8) -- (13) ;
\draw (13) -- (18) ;
\draw (18) -- (23) ;
\draw (23) -- (28) ;
\draw (4) -- (9) ;
\draw (9) -- (14) ;
\draw (14) -- (19) ;
\draw (19) -- (24) ;
\draw (24) -- (29) ;
\draw (5) -- (10) ;
\draw (10) -- (15) ;
\draw (15) -- (20) ;
\draw (20) -- (25) ;
\draw (25) -- (30) ;
\draw (5) -- (34) ;
\draw (31) -- (35) ;
\draw (32) -- (36) ;
\draw (33) -- (37) ;
\draw (30) -- (38) ;
\draw (34) -- (35) ;
\draw (35) -- (36) ;
\draw (36) -- (37) ;
\draw (37) -- (38) ;

\coordinate (39) at ($(21)!0.5!(27)$) ;
\draw (39) node{\large 1} ;
\coordinate (40) at ($(22)!0.5!(28)$) ;
\draw (40) node{\large 2} ;
\coordinate (41) at ($(23)!0.5!(29)$) ;
\draw (41) node{\large 3} ;
\coordinate (42) at ($(24)!0.5!(30)$) ;
\draw (42) node{\large 4} ;
\coordinate (43) at ($(16)!0.5!(22)$) ;
\draw (43) node{\large 5} ; 
\coordinate (44) at ($(17)!0.5!(23)$) ;
\draw (44) node{\large 6} ;
\coordinate (45) at ($(18)!0.5!(24)$) ;
\draw (45) node{\large 7} ;
\coordinate (46) at ($(19)!0.5!(25)$) ;
\draw (46) node{\large 8} ;
\coordinate (47) at ($(11)!0.5!(17)$) ;
\draw (47) node{\large 9} ;
\coordinate (48) at ($(12)!0.5!(18)$) ;
\draw (48) node{\large 10} ;
\coordinate (49) at ($(13)!0.5!(19)$) ;
\draw (49) node{\large 11} ;
\coordinate (50) at ($(14)!0.5!(20)$) ;
\draw (50) node{\large 12} ;
\coordinate (51) at ($(6)!0.5!(12)$) ;
\draw (51) node{\large 13} ;
\coordinate (52) at ($(7)!0.5!(13)$) ;
\draw (52) node{\large 14} ;
\coordinate (53) at ($(8)!0.5!(14)$) ;
\draw (53) node{\large 15} ;
\coordinate (54) at ($(9)!0.5!(15)$) ;
\draw (54) node{\large 16} ;
\coordinate (55) at ($(1)!0.5!(7)$) ;
\draw (55) node{\large 17} ;
\coordinate (56) at ($(2)!0.5!(8)$) ;
\draw (56) node{\large 18} ;
\coordinate (57) at ($(3)!0.5!(9)$) ;
\draw (57) node{\large 19} ;
\coordinate (58) at ($(4)!0.5!(10)$) ;
\draw (58) node{\large 20} ;
\coordinate (59) at ($(33)!0.5!(38)$) ;
\draw (59) node{\large 21} ;
\coordinate (60) at ($(32)!0.5!(37)$) ;
\draw (60) node{\large 22} ;
\coordinate (61) at ($(31)!0.5!(36)$) ;
\draw (61) node{\large 23} ;
\coordinate (62) at ($(5)!0.5!(35)$) ;
\draw (62) node{\large 24} ;

\draw[decorate,decoration={brace,raise=3pt,amplitude=4pt}] (1)--(6) node[midway,xshift=-0.6cm]{\Large $\frac{1}{5}$} ;
\draw[decorate,decoration={brace,raise=3pt,amplitude=4pt}] (6) -- (11) node[midway,xshift=-0.6cm]{\Large $\frac{1}{5}$} ;
\draw[decorate,decoration={brace,raise=3pt,amplitude=4pt}] (11) -- (16) node[midway,xshift=-0.6cm]{\Large $\frac{1}{5}$} ;
\draw[decorate,decoration={brace,raise=3pt,amplitude=4pt}] (16) -- (21) node[midway,xshift=-0.6cm]{\Large $\frac{1}{5}$} ;
\draw[decorate,decoration={brace,raise=3pt,amplitude=4pt}] (21) -- (26) node[midway,xshift=-0.6cm]{\Large $\frac{1}{5}$} ;
\draw[decorate,decoration={brace,raise=3pt,amplitude=4pt,mirror}] (1) -- (2) node[midway,yshift=-0.6cm]{\Large $\frac{21}{100}$} ;
\draw[decorate,decoration={brace,raise=3pt,amplitude=4pt,mirror}] (2) -- (3) node[midway,yshift=-0.6cm]{\Large $\frac{21}{100}$} ;
\draw[decorate,decoration={brace,raise=3pt,amplitude=4pt,mirror}] (3) -- (4) node[midway,yshift=-0.6cm]{\Large $\frac{21}{100}$} ;
\draw[decorate,decoration={brace,raise=3pt,amplitude=4pt,mirror}] (4) -- (5) node[midway,yshift=-0.6cm]{\Large $\frac{21}{100}$} ;
\draw[decorate,decoration={brace,raise=3pt,amplitude=4pt,mirror}] (5) -- (34) node[midway,yshift=-0.6cm]{\Large $\frac{4}{25}$} ;
\draw[decorate,decoration={brace,raise=3pt,amplitude=4pt,mirror}] (34) -- (35) node[midway,xshift=0.6cm]{\Large $\frac{1}{4}$} ;
\draw[decorate,decoration={brace,raise=3pt,amplitude=4pt,mirror}] (35) -- (36) node[midway,xshift=0.6cm]{\Large $\frac{1}{4}$} ;
\draw[decorate,decoration={brace,raise=3pt,amplitude=4pt,mirror}] (36) -- (37) node[midway,xshift=0.6cm]{\Large $\frac{1}{4}$} ;
\draw[decorate,decoration={brace,raise=3pt,amplitude=4pt,mirror}] (37) -- (38) node[midway,xshift=0.6cm]{\Large $\frac{1}{4}$} ;

\end{tikzpicture}
\end{center}

Below is a picture for the $t = 5$ case with the labels $\frac{1}{t-1},\frac{1}{t},\frac{1}{t}-\frac{1}{t^2}$, and $\frac{1}{t-1}-\frac{1}{t^2}$. \\

\begin{center}
\begin{tikzpicture}[declare function={a=5;}]

\draw (0,0) coordinate (1) ;
\draw (21*a/100,0) coordinate (2) ;
\draw (21*a/50,0) coordinate (3) ;
\draw (63*a/100,0) coordinate (4) ;
\draw (21*a/25,0) coordinate (5) ;
\draw (0,a/5) coordinate (6) ;
\draw (21*a/100,a/5) coordinate (7) ;
\draw (21*a/50,a/5) coordinate (8) ;
\draw (63*a/100,a/5) coordinate (9) ;
\draw (21*a/25,a/5) coordinate (10) ;
\draw (0,2*a/5) coordinate (11) ;
\draw (21*a/100,2*a/5) coordinate (12) ;
\draw (21*a/50,2*a/5) coordinate (13) ;
\draw (63*a/100,2*a/5) coordinate (14) ;
\draw (21*a/25,2*a/5) coordinate (15) ;
\draw (0,3*a/5) coordinate (16) ;
\draw (21*a/100,3*a/5) coordinate (17) ;
\draw (21*a/50,3*a/5) coordinate (18) ;
\draw (63*a/100,3*a/5) coordinate (19) ;
\draw (21*a/25,3*a/5) coordinate (20) ;
\draw (0,4*a/5) coordinate (21) ;
\draw (21*a/100,4*a/5) coordinate (22) ;
\draw (21*a/50,4*a/5) coordinate (23) ;
\draw (63*a/100,4*a/5) coordinate (24) ;
\draw (21*a/25,4*a/5) coordinate (25) ;
\draw (0,a) coordinate (26) ;
\draw (21*a/100,a) coordinate (27) ;
\draw (21*a/50,a) coordinate (28) ;
\draw (63*a/100,a) coordinate (29) ;
\draw (21*a/25,a) coordinate (30) ;
\draw (21*a/25,a/4) coordinate (31) ;
\draw (21*a/25,a/2) coordinate (32) ;
\draw (21*a/25,3*a/4) coordinate (33) ;
\draw (a,0) coordinate (34) ;
\draw (a,a/4) coordinate (35) ;
\draw (a,a/2) coordinate (36) ;
\draw (a,3*a/4) coordinate (37) ;
\draw (a,a) coordinate (38) ;

\draw (1) -- (2) ;
\draw (2) -- (3) ;
\draw (3) -- (4) ;
\draw (4) -- (5) ;
\draw (6) -- (7) ;
\draw (7) -- (8) ;
\draw (8) -- (9) ;
\draw (9) -- (10) ;
\draw (11) -- (12) ;
\draw (12) -- (13) ;
\draw (13) -- (14) ;
\draw (14) -- (15) ;
\draw (16) -- (17) ;
\draw (17) -- (18) ;
\draw (18) -- (19) ;
\draw (19) -- (20) ;
\draw (21) -- (22) ;
\draw (22) -- (23) ;
\draw (23) -- (24) ;
\draw (24) -- (25) ;
\draw (26) -- (27) ;
\draw (27) -- (28) ;
\draw (28) -- (29) ;
\draw (29) -- (30) ;
\draw (1) -- (6) ;
\draw (6) -- (11) ;
\draw (11) -- (16) ;
\draw (16) -- (21) ;
\draw (21) -- (26) ;
\draw (2) -- (7) ;
\draw (7) -- (12) ;
\draw (12) -- (17) ;
\draw (17) -- (22) ;
\draw (22) -- (27) ;
\draw (3) -- (8) ;
\draw (8) -- (13) ;
\draw (13) -- (18) ;
\draw (18) -- (23) ;
\draw (23) -- (28) ;
\draw (4) -- (9) ;
\draw (9) -- (14) ;
\draw (14) -- (19) ;
\draw (19) -- (24) ;
\draw (24) -- (29) ;
\draw (5) -- (10) ;
\draw (10) -- (15) ;
\draw (15) -- (20) ;
\draw (20) -- (25) ;
\draw (25) -- (30) ;
\draw (5) -- (34) ;
\draw (31) -- (35) ;
\draw (32) -- (36) ;
\draw (33) -- (37) ;
\draw (30) -- (38) ;
\draw (34) -- (35) ;
\draw (35) -- (36) ;
\draw (36) -- (37) ;
\draw (37) -- (38) ;

\coordinate (39) at ($(21)!0.5!(27)$) ;
\draw (39) node{\large 1} ;
\coordinate (40) at ($(22)!0.5!(28)$) ;
\draw (40) node{\large 2} ;
\coordinate (41) at ($(23)!0.5!(29)$) ;
\draw (41) node{\large 3} ;
\coordinate (42) at ($(24)!0.5!(30)$) ;
\draw (42) node{\large 4} ;
\coordinate (43) at ($(16)!0.5!(22)$) ;
\draw (43) node{\large 5} ; 
\coordinate (44) at ($(17)!0.5!(23)$) ;
\draw (44) node{\large 6} ;
\coordinate (45) at ($(18)!0.5!(24)$) ;
\draw (45) node{\large 7} ;
\coordinate (46) at ($(19)!0.5!(25)$) ;
\draw (46) node{\large 8} ;
\coordinate (47) at ($(11)!0.5!(17)$) ;
\draw (47) node{\large 9} ;
\coordinate (48) at ($(12)!0.5!(18)$) ;
\draw (48) node{\large 10} ;
\coordinate (49) at ($(13)!0.5!(19)$) ;
\draw (49) node{\large 11} ;
\coordinate (50) at ($(14)!0.5!(20)$) ;
\draw (50) node{\large 12} ;
\coordinate (51) at ($(6)!0.5!(12)$) ;
\draw (51) node{\large 13} ;
\coordinate (52) at ($(7)!0.5!(13)$) ;
\draw (52) node{\large 14} ;
\coordinate (53) at ($(8)!0.5!(14)$) ;
\draw (53) node{\large 15} ;
\coordinate (54) at ($(9)!0.5!(15)$) ;
\draw (54) node{\large 16} ;
\coordinate (55) at ($(1)!0.5!(7)$) ;
\draw (55) node{\large 17} ;
\coordinate (56) at ($(2)!0.5!(8)$) ;
\draw (56) node{\large 18} ;
\coordinate (57) at ($(3)!0.5!(9)$) ;
\draw (57) node{\large 19} ;
\coordinate (58) at ($(4)!0.5!(10)$) ;
\draw (58) node{\large 20} ;
\coordinate (59) at ($(33)!0.5!(38)$) ;
\draw (59) node{\large 21} ;
\coordinate (60) at ($(32)!0.5!(37)$) ;
\draw (60) node{\large 22} ;
\coordinate (61) at ($(31)!0.5!(36)$) ;
\draw (61) node{\large 23} ;
\coordinate (62) at ($(5)!0.5!(35)$) ;
\draw (62) node{\large 24} ;

\draw[decorate,decoration={brace,raise=3pt,amplitude=4pt}] (1)--(6) node[midway,xshift=-0.6cm]{\Large $\frac{1}{t}$} ;

\draw[decorate,decoration={brace,raise=3pt,amplitude=4pt,mirror}] (1) -- (2) node[midway,yshift=-0.6cm]{\Large $\frac{1}{t-1}-\frac{1}{t^2}$} ;
\draw[decorate,decoration={brace,raise=3pt,amplitude=4pt,mirror}] (5) -- (34) node[midway,yshift=-0.6cm]{\Large $\frac{1}{t}- \frac{1}{t^2}$} ;
\draw[decorate,decoration={brace,raise=3pt,amplitude=4pt,mirror}] (34) -- (35) node[midway,xshift=0.6cm]{\Large $\frac{1}{t-1}$} ;
\end{tikzpicture}
\end{center}

\minus*

\begin{proof}
To see that $g^*(t^2-1) \leq \frac{1}{t} - \frac{1}{t^2} + \frac{1}{t-1}$, see the construction described above. More formally, define a coloring of the unit square $\chi : [0,1]^2 \rightarrow [t^2-1]$ where for $i \in \{0\} \cup [t-2]$ and $j \in \{0\} \cup [t-1]$, for all 

$$(x,y) \in \left( i \cdot (\frac{1}{t-1} - \frac{1}{t^2}), \ (i+1) \cdot(\frac{1}{t-1} - \frac{1}{t^2}) \right) \times \left( j \cdot \frac{1}{t}, \ (j+1) \cdot \frac{1}{t} \right) $$

\noindent
set $\chi(x,y) = (t-j-1) \cdot (t-1) + i + 1$. For $j \in \{0\} \cup [t-2] $, for all 

$$(x,y) \in \left( 1- (\frac{1}{t}-\frac{1}{t^2}), \ 1 \right) \times \left( j \cdot \frac{1}{t}, \ (j+1) \cdot \frac{1}{t-1} \right) $$

\noindent 
set $\chi(x,y) = t^2 - j - 1$. For all the other points we haven't colored yet, we can choose any color we like. It doesn't matter because their projections onto the $x$ and $y$ axes have Lebesgue measure $0$.  \\

Now we show $g^*(t^2-1) \geq \frac{1}{t} - \frac{1}{t^2} + \frac{1}{t-1}$. We may assume $a_i+b_i \leq \frac{1}{t}-\frac{1}{t^2} + \frac{1}{t-1}$ for all $i \in [t^2-1]$. Let $k = \frac{1}{t} - \frac{1}{t^2} + \frac{1}{t-1}$. \\

\textbf{Claim 1}: if $R \subseteq [t^2-1]$ and $|R| \leq t-2$, then $a(R) = b(R) = 0$. Assume for the sake of contradiction otherwise. Note we can't have that $a(\{\}) > 0$, because by constraint (5) this would imply that $b(R) = 0$ for all $R \subseteq [t^2-1]$, contradicting constraint (2). Likewise, we can't have $b(\{\}) > 0$. So we may assume $|R| \geq 1$. We have $a(R) > 0$ implies that $\sum_{i \in R} b_i \geq 1$. We apply Lemma 3 with $s = |R|$, which we can do because 

$$k = \frac{1}{t} - \frac{1}{t^2} + \frac{1}{t-1} \leq \frac{2}{t-2} < \frac{2}{|R|}.$$

\noindent
The inequality $\frac{1}{t} - \frac{1}{t^2} + \frac{1}{t-1} \leq \frac{2}{t-2}$ holds when $t \geq 3$, which can be checked by algebraic manipulations or by mathematical software (we used Mathematica). \\

Thus Lemma 3 says $\sum_{i \in R} a_ib_i \leq k - \frac{1}{|R|}$, and so we must have 

$$1-k+\frac{1}{|R|} \leq \sum_{i \in [t^2-1] \setminus R} a_ib_i \leq (t^2-1-|R|)\frac{k^2}{4}, $$

\noindent
where the last inequality comes from AM-GM. The inequality $1-k+\frac{1}{|R|} \leq (t^2-1-|R|)\frac{k^2}{4}$ is a contradiction when $t \geq 3$ and $1 \leq |R| \leq t-2$ which can be checked by mathematical software. Thus Claim 1 is proven. \\

\textbf{Claim 2}: There do not exist $R_1,R_2 \subseteq [t^2-1]$ such that $|R_1| = |R_2| = t-1$ and $a(R_1) > 0, b(R_2) > 0$. Assume otherwise. By constraint (5), we have $R_1 \cap R_2 \neq \{\}$. First assume $R_1 \setminus R_2 \neq \{\}$. From $a(R_1) > 0$, we have $\sum_{i \in R_1} b_i \geq 1$ and from $b(R_2) > 0$, we have $\sum_{i \in R_2} a_i \geq 1$. We wish to apply Lemma 6 with $I = R_1 \cap R_2$, $O_1 = R_1 \setminus R_2$, and $O_2 = R_2 \setminus R_1$. The inequality $k < \frac{2}{t-1} = \frac{2}{|I|+|O_1|}$ holds for $t \geq 3$ which can be checked by mathematical software. Thus by Lemma 6 for $e = \frac{2-|I|k}{2|O_1|}$ we get

$$\sum_{i \in R_1 \cup R_2} a_ib_i \leq |I|\frac{k^2}{4} + 2|O_1| \cdot e(k-e)$$

\noindent 
so that

\begin{align*}
1-(|I|\frac{k^2}{4} + 2|O_1| \cdot e(k-e)) &\leq \sum_{i \in [t^2-1] \setminus (R_1 \cup R_2)} a_ib_i \\
&\leq (t^2 - 1 - |I| -|O_1|-|O_2|)\frac{k^2}{4}.
\end{align*}

The inequality 

$$1-(|I|\frac{k^2}{4} + 2|O_1| \cdot e(k-e)) \leq (t^2 - 1 - |I| -2|O_1|)\frac{k^2}{4}$$

\noindent
reduces to 

$$1-(|I|\frac{k^2}{4} + 2(t-1-|I|) \cdot e(k-e)) \leq (t^2 - 1 - |I| -2(t-1-|I|))\frac{k^2}{4}$$

\noindent
where $e = \frac{2-|I|k}{2(t-1-|I|)}$ because $|I| + |O_1| = t-1$. This last inequality is a contradiction for $t \geq 3$ and $1 \leq |I| \leq t-2$ which can be checked by mathematical software.  \\

Now we consider the case when $O_1 = O_2 = \{\}$, and thus $R_1 = R_2$. From $a(R_1) > 0$ we get $\sum_{i \in R_1} b_i \geq 1$ and from $b(R_1) > 0$ we get $\sum_{i \in R_1} a_i \geq 1$. We sum these two equations to get $\sum_{i \in R_1} a_i+b_i \geq 2$, so that for some $i \in R_1$ we have $a_i+b_i \geq \frac{2}{|R_1|} = \frac{2}{t-1} > k$, a contradiction. Thus Claim 2 is proven. \\

\textbf{Claim 3}: There do not exist $R_1,R_2 \subseteq [t^2-1]$ such that $|R_1| = |R_2| = t-1$, $R_1 \neq R_2$, and $a(R_1), a(R_2) > 0$ or $b(R_1), b(R_2) > 0$. Assume otherwise, we show it for $a(R_1), a(R_2) > 0$. From $a(R_1) > 0$, we have $\sum_{i \in R_1} b_i \geq 1$ and from $a(R_2) > 0$ we have $\sum_{i \in R_2} b_i \geq 1$. We wish to apply Lemma 7 with $I = R_1 \cap R_2$, $O_1 = R_1 \setminus R_2$, and $O_2 = R_2 \setminus R_1$. We can do this because $k < \frac{2}{|I|+|O_1|} = \frac{2}{t-1}$. Thus Lemma 7 gives us for $e_1 = \frac{2+|I|k}{2|I|+|O_1|} - \frac{k}{2}$ and $e_2 = \frac{1+|I|\frac{k}{2}}{2|I| + |O_1|}$ that

$$\sum_{i \in R_1 \cup R_2} a_ib_i \leq |I| \cdot e_1(k-e_1) + 2|O_1| \cdot e_2(k - e_2)$$

and thus

\begin{align*}
1 - (|I| \cdot e_1(k-e_1) + 2|O_1| \cdot e_2(k - e_2)) &\leq \sum_{i \in [t^2-1] \setminus R_1 \cup R_2} a_ib_i \\
&\leq (t^2 - 1 - |I| -2|O_1|)\frac{k^2}{4}.
\end{align*}

The inequality

$$1 - (|I| \cdot e_1(k-e_1) + 2|O_1| \cdot e_2(k - e_2)) \leq (t^2 - 1 - |I| -2|O_1|)\frac{k^2}{4}$$

\noindent
reduces to

$$1 - (|I| \cdot e_1(k-e_1) + 2(t-1-|I|) \cdot e_2(k - e_2)) \leq (t^2 - 1 - |I| -2(t-1-|I|))\frac{k^2}{4}$$

\noindent
where $e_1 = \frac{2+|I|k}{t-1+|I|} - \frac{k}{2}$ and $e_2 = \frac{1+|I|\frac{k}{2}}{t-1+|I|}$ because $|I| + |O_1| = t-1$. This last inequality is a contradiction for $t \geq 3$ and $0 \leq |I| \leq t-2$ which can be checked by mathematical software. Thus Claim 3 is proven. \\

By Claims 1 and 2, we must have either $a(R) > 0$ implies $|R|  \geq t$ for all $R \subseteq [t^2-1]$, or $b(R) > 0$ implies $|R| \geq t$ for all $R \subseteq [t^2-1]$, since we know there can't exist $R_1,R_2 \subseteq [t^2-1]$ such that $|R_1|,|R_2| \leq t-1$ and $a(R_1) > 0$ and $b(R_2) > 0$. Without loss of generality assume that if $b(R) > 0$, then $|R| \geq t$ for all $R \subseteq [t^2-1]$. Since $a_i + b_i \leq k$ for all $i \in [t^2-1]$, we have that 

$$\sum_{R \subseteq [t^2-1]} |R|(a(R) + b(R)) = \sum_{i \in [t^2-1]} a_i + b_i \leq (t^2-1)k$$

\noindent
so that

\begin{align*}
(t^2-1)k &\geq \sum_{R \subseteq [t^2-1]} (a(R) + b(R))|R| = \sum_{R \subseteq [t^2-1]} |R| \cdot a(R) \ + \sum_{R \subseteq [t^2-1], |R| \geq t} |R| \cdot b(R) \\
&\geq \sum_{R \subseteq [t^2-1]} |R| \cdot a(R) + \sum_{R \subseteq [t^2-1], |R| \geq t} |R| \cdot b(R)  = \sum_{R \subseteq [t^2-1]} |R| \cdot a(R) + t
\end{align*}

\noindent
and thus $\sum_{R \subseteq [t^2-1]} |R| \cdot a(R) \leq (t^2-1)k - t$. In particular, this implies that 

\begin{align*}
(t^2-1)k - t &\geq \sum_{R \subseteq [t^2-1]} |R| \cdot a(R) \\
&= \left( \sum_{R \subseteq [t^2-1], |R| = t-1} |R| \cdot a(R) \right) + \left( \sum_{R \subseteq [t^2-1], |R| \geq t-1} |R| \cdot a(R) \right)  \\
&\geq (t-1)\left( \sum_{R \subseteq [t^2-1], |R| = t-1} a(R) \right) + t\left( \sum_{R \subseteq [t^2-1], |R| \geq t} a(R) \right) \\
&= t-1+ \left( \sum_{R \subseteq [t^2-1], |R| \geq t} a(R) \right)
\end{align*}

\noindent
so that $\sum_{R \subseteq [t^2-1], |R| \geq t} a(R) \leq (t^2-1)k-2t+1$ and thus 

\begin{align*}
\sum_{R \subseteq [t^2-1], |R| = t-1} a(R) \cdot |R| &\geq 1-((t^2-1)k-2t+1) \\
&= 1-((t^2-1)(\frac{1}{t} - \frac{1}{t^2} + \frac{1}{t-1})-2t+1) \\
&= 1-(\frac{t^2-1}{t} - \frac{t^2-1}{t^2} + t+1-2t+1) \\
&= 1 - (\frac{t^3-t-t^2+1 - t^3 + 2t^2}{t^2}) \\
&=  1 - (\frac{-t+1 + t^2}{t^2}) = \frac{1}{t} - \frac{1}{t^2} > 0.
\end{align*}

Thus there exists $R \subseteq [t^2-1]$ where $|R| = t-1$ and $a(R) > 0$. Without loss of generality, let $R = [t-1]$. By Claim 3 $R$ is a unique such subset. Thus

$$\sum_{R \subseteq [t^2-1], |R| = t-1} a(R) =  a([t-1]) \geq \frac{1}{t} - \frac{1}{t^2}.$$

\noindent
From $a([t-1]) \geq \frac{1}{t} - \frac{1}{t^2}$ we get $a_i \geq \frac{1}{t} - \frac{1}{t^2}$ for all $i \in [t-1]$, and thus $\sum_{i \in [t-1]} a_i \geq (t-1)(\frac{1}{t} - \frac{1}{t^2})$. From $a([t-1]) > 0$ we get $\sum_{i \in [t-1]} b_i \geq 1$. Summing these two inequalities together we get $\sum_{i \in [t-1]} a_i+b_i \geq (t-1)(\frac{1}{t} - \frac{1}{t^2}) + 1$, and thus for some $j \in [t-1]$ we have

$$a_j+b_j \geq \frac{1}{t} - \frac{1}{t^2} + \frac{1}{t-1}$$

\noindent
as required. 
\end{proof}

\begin{center}
\section{Universal Upper Bound on $g^*(r)$}
\end{center}

We followed a specific procedure in constructing the coloring squares in the propositions and theorems above. We conjecture that this procedure produces an optimal solution for all $r \in \mathbb{Z}^{+}$. We describe the procedure for constructing coloring squares below. \\

We show the procedure for constructing coloring squares on the example $r = 28$. We first note that $5^2 \leq 28 \leq (5+1)^2$ and that $5^2 + 3 = 28 \leq 5 \cdot (5+1)$. We will take the $5$-tuple $(5,5,6,6,6) \in \{5,6\}^{5}$, and note that $5 + 5 + 6 + 6 + 6 = 28$. We will partition the square into five ``columns" such that the first two columns have the same horizontal side length and the last three columns have the same horizontal side length. Then we will further partition the each of the first two columns into 5 equally sized rectangles stacked on top of each other and the last three columns into 6 equally sized rectangles stacked on top of each other as shown below: \\

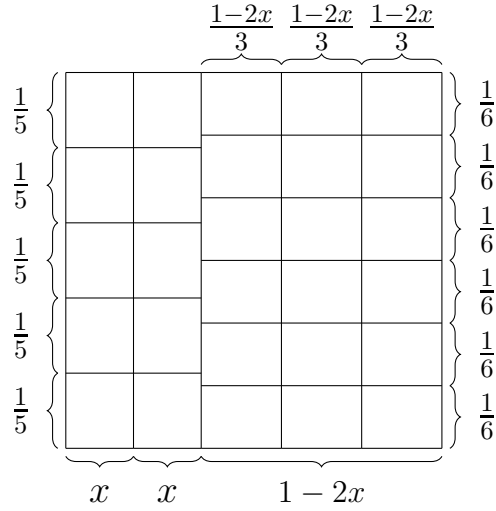
\begin{center}
\begin{tikzpicture}[declare function={a=5;}]

\draw (0,0) coordinate (1) ;
\draw (9*a/50,0) coordinate (2) ;
\draw (9*a/25,0) coordinate (3) ;
\draw (0,a/5) coordinate (4) ;
\draw (9*a/50,a/5) coordinate (5) ;
\draw (9*a/25,a/5) coordinate (6) ;
\draw (0,2*a/5) coordinate (7) ;
\draw (9*a/50,2*a/5) coordinate (8) ;
\draw (9*a/25,2*a/5) coordinate (9) ;
\draw (0,3*a/5) coordinate (10) ;
\draw (9*a/50,3*a/5) coordinate (11) ;
\draw (9*a/25,3*a/5) coordinate (12) ;
\draw (0,4*a/5) coordinate (13) ;
\draw (9*a/50,4*a/5) coordinate (14) ;
\draw (9*a/25,4*a/5) coordinate (15) ;
\draw (0,a) coordinate (16) ;
\draw (9*a/50,a) coordinate (17) ;
\draw (9*a/25,a) coordinate (18) ;
\draw (43*a/75,0) coordinate (19) ;
\draw (59*a/75,0) coordinate (20) ;
\draw (a,0) coordinate (21) ;
\draw (9*a/25,a/6) coordinate (22) ;
\draw (43*a/75,a/6) coordinate (23) ;
\draw (59*a/75,a/6) coordinate (24) ;
\draw (a,a/6) coordinate (25) ;
\draw (9*a/25,a/3) coordinate (26) ;
\draw (43*a/75,a/3) coordinate (27) ;
\draw (59*a/75,a/3) coordinate (28) ;
\draw (a,a/3) coordinate (29) ;
\draw (9*a/25,a/2) coordinate (30) ;
\draw (43*a/75,a/2) coordinate (31) ;
\draw (59*a/75,a/2) coordinate (32) ;
\draw (a,a/2) coordinate (33) ;
\draw (9*a/25,2*a/3) coordinate (34) ;
\draw (43*a/75,2*a/3) coordinate (35) ;
\draw (59*a/75,2*a/3) coordinate (36) ;
\draw (a,2*a/3) coordinate (37) ;
\draw (9*a/25,5*a/6) coordinate (38) ;
\draw (43*a/75,5*a/6) coordinate (39) ;
\draw (59*a/75,5*a/6) coordinate (40) ;
\draw (a,5*a/6) coordinate (41) ;
\draw (43*a/75,a) coordinate (42) ;
\draw (59*a/75,a) coordinate (43) ;
\draw (a,a) coordinate (44) ;

\draw (1) -- (2) ;
\draw (2) -- (3) ;
\draw (3) -- (19) ;
\draw (19) -- (20) ;
\draw (20) -- (21) ;
\draw (4) -- (5) ;
\draw (5) -- (6) ;
\draw (22) -- (23) -- (24) -- (25) ;
\draw (7) -- (8) -- (9) ;
\draw (26) -- (27) -- (28) -- (29) ;
\draw (30) -- (31) -- (32) -- (33) ;
\draw (34) -- (37) ;
\draw (38) -- (41) ;
\draw (18) -- (44) ;
\draw (10) -- (12) ;
\draw (13) -- (15) ;
\draw (16) -- (18) ;
\draw (1) -- (16) ;
\draw (2) -- (17) ;
\draw (3) -- (18) ;
\draw (19) -- (42) ;
\draw (20) -- (43) ;
\draw (21) -- (44) ;

\draw[decorate,decoration={brace,raise=3pt,amplitude=4pt}] (1) -- (4) node[midway,xshift=-0.6cm]{\Large $\frac{1}{5}$} ;
\draw[decorate,decoration={brace,raise=3pt,amplitude=4pt}] (4) -- (7) node[midway,xshift=-0.6cm]{\Large $\frac{1}{5}$} ;
\draw[decorate,decoration={brace,raise=3pt,amplitude=4pt}] (7) -- (10) node[midway,xshift=-0.6cm]{\Large $\frac{1}{5}$} ;
\draw[decorate,decoration={brace,raise=3pt,amplitude=4pt}] (10) -- (13) node[midway,xshift=-0.6cm]{\Large $\frac{1}{5}$} ;
\draw[decorate,decoration={brace,raise=3pt,amplitude=4pt}] (13) -- (16) node[midway,xshift=-0.6cm]{\Large $\frac{1}{5}$} ;
\draw[decorate,decoration={brace,raise=3pt,amplitude=4pt,mirror}] (1) -- (2) node[midway,yshift=-0.6cm]{\Large $x$} ;
\draw[decorate,decoration={brace,raise=3pt,amplitude=4pt,mirror}] (2) -- (3) node[midway,yshift=-0.6cm]{\Large $x$} ;
\draw[decorate,decoration={brace,raise=3pt,amplitude=4pt,mirror}] (3) -- (21) node[midway,yshift=-0.6cm]{\large $1-2x$} ;
\draw[decorate,decoration={brace,raise=3pt,amplitude=4pt,mirror}] (21) -- (25) node[midway,xshift=0.6cm]{\Large $\frac{1}{6}$} ;
\draw[decorate,decoration={brace,raise=3pt,amplitude=4pt,mirror}] (25) -- (29) node[midway,xshift=0.6cm]{\Large $\frac{1}{6}$} ;
\draw[decorate,decoration={brace,raise=3pt,amplitude=4pt,mirror}] (29) -- (33) node[midway,xshift=0.6cm]{\Large $\frac{1}{6}$} ;
\draw[decorate,decoration={brace,raise=3pt,amplitude=4pt,mirror}] (33) -- (37) node[midway,xshift=0.6cm]{\Large $\frac{1}{6}$} ;
\draw[decorate,decoration={brace,raise=3pt,amplitude=4pt,mirror}] (37) -- (41) node[midway,xshift=0.6cm]{\Large $\frac{1}{6}$} ;
\draw[decorate,decoration={brace,raise=3pt,amplitude=4pt,mirror}] (41) -- (44) node[midway,xshift=0.6cm]{\Large $\frac{1}{6}$} ;
\draw[decorate,decoration={brace,raise=3pt,amplitude=4pt}] (18) -- (42) node[midway,yshift=0.6cm]{\Large $\frac{1-2x}{3}$} ;
\draw[decorate,decoration={brace,raise=3pt,amplitude=4pt}] (42) -- (43) node[midway,yshift=0.6cm]{\Large $\frac{1-2x}{3}$} ;
\draw[decorate,decoration={brace,raise=3pt,amplitude=4pt}] (43) -- (44) node[midway,yshift=0.6cm]{\Large $\frac{1-2x}{3}$} ;

\end{tikzpicture}
\end{center}

\noindent
$x$ is the rational number satisfying $\frac{1}{6} + x = \frac{1}{5} + \frac{1-3x}{2}$ so that all of the rectangles in each column have the same sum of side lengths. \\

In general, given $t^2 \leq r \leq t(t+1)$ and $r = t^2 + q$, we will partition the square into $t$ columns where the first $p = t-q$ columns have the same horizontal side length and the last $q$ columns have the same horizontal side length. We will further partition each of the first $p$ columns into $t$ equally sized rectangles stacked on top of each other, and we will further partition each of the last $q$ columns into $t+1$ equally sized rectangles stacked on top of each other. A rectangle inside one of the first $p$ columns will have horizontal side length $x$ and vertical side length $\frac{1}{t}$, and a rectangle inside one of the last $q$ columns will have horizontal side length $\frac{1-px}{q}$ and vertical side length $\frac{1}{t+1}$, where $x$ satisfies $\frac{1}{t} + x = \frac{1}{t+1} + \frac{1-px}{q}$, and so this way rectangles in any column have the same side length. \\

In the case that $t(t+1) \leq r \leq (t+1)^2$, and $r = (t+1)^2-q$, we will partition the square into $t+1$ columns where the first $p = t+1-q$ columns have the same horizontal side length and the last $q$ columns have the same horizontal side length. We will further partition each of the first $p$ columns into $t+1$ equally sized rectangles stacked on top of each other, and we will further partition each of the last $q$ columns into $t$ equally sized rectangles stacked on top of each other. A rectangle inside one of the first $p$ columns will have horizontal side length $x$ and vertical side length $\frac{1}{t+1}$, and a rectangle inside one of the last $q$ columns will have horizontal side length $\frac{1-px}{q}$ and vertical side length $\frac{1}{t}$, where $x$ satisfies $\frac{1}{t+1} + x = \frac{1}{t} + \frac{1-px}{q}$, and so this way rectangles in any column have the same side length. \\

The theorems below are a result of following the procedure described above and finding $x$ that optimizes the side lengths of the smaller rectangles. \\

\univone*

\begin{proof}
We define a coloring $\chi : [0,1]^2 \rightarrow [r]$ which satisfies 

$$a_i + b_i = \frac{2}{t} + \frac{1}{t+1} + \frac{r}{t(t+1)} - \frac{r}{t^2}$$

\noindent 
for all $i \in [r]$; it will be the construction described above.  Let $p = t(t+1)-r$ and $p+q = t$. Set $x = \frac{p}{t^2}+\frac{q}{t(t+1)}$. Define $[0] = \{\}$. For $i \in \{0\} \cup [p-1]$ and $j \in \{0\} \cup [t-1]$, for all

$$(x,y) \in \biggl( i \cdot x, \ (i+1) \cdot x \biggl) \times \left( j \cdot \frac{1}{t}, \ (j+1) \cdot \frac{1}{t} \right)$$

\noindent
set $\chi(x,y) = (t-j-1) \cdot p + i + 1$.  If $p = 0$, then we skip this step. For $i \in \{0\} \cup [q-1]$ and $j \in \{0\} \cup [t]$, for all 

$$(x,y) \in \biggl(i \cdot \frac{1-px}{q} + px, \ (i+1) \cdot \frac{1-px}{q} + px \biggl) \times \biggl(j \cdot \frac{1}{t+1}, \ (j+1) \cdot \frac{1}{t+1} \biggl)$$

\noindent
set $\chi(x,y) = (t-j) \cdot q + i + 1 + pt$. If $q = 0$ we skip this step. For all the points we haven't colored yet, we can choose any color we want. Their projections on the $x$ and $y$ axes have Lebesgue measure 0.  One can check $a_i + b_i = \frac{2}{t} + \frac{1}{t+1} + \frac{r}{t(t+1)} - \frac{r}{t^2}$ for all $i \in [r]$. 

\end{proof}

\univtwo*

\begin{proof}
We define a coloring $\chi : [0,1]^2 \rightarrow [r]$ which satisfies 

$$a_i + b_i = \frac{2}{t+1}+\frac{1}{t}+\frac{r}{(t+1)^2}-\frac{r}{t(t+1)}$$

\noindent 
for all $i \in [r]$; it will be the construction described above.  Let $p = r-t(t+1)$ and $p+q = t+1$. Set $x = \frac{p}{(t+1)^2}+\frac{q}{t(t+1)}$. Define $[0] = \{\}$. For $i \in \{0\} \cup [p-1]$ and $j \in \{0\} \cup [t]$, for all

$$(x,y) \in \biggl( i \cdot x, \ (i+1) \cdot x \biggl) \times \left( j \cdot \frac{1}{t+1}, \ (j+1) \cdot \frac{1}{t+1} \right)$$

\noindent
set $\chi(x,y) = (t-j) \cdot p + i + 1$.  If $p = 0$, then we skip this step. For $i \in \{0\} \cup [q-1]$ and $j \in \{0\} \cup [t-1]$, for all 

$$(x,y) \in \biggl(i \cdot \frac{1-px}{q} + px, \ (i+1) \cdot \frac{1-px}{q} + px \biggl) \times \biggl(j \cdot \frac{1}{t}, \ (j+1) \cdot \frac{1}{t} \biggl)$$

\noindent
set $\chi(x,y) = (t-j-1) \cdot q + i + 1 + p(t+1)$. If $q = 0$ we skip this step. For all the points we haven't colored yet, we can choose any color we want. Their projections on the $x$ and $y$ axes have Lebesgue measure 0.  One can check $a_i + b_i = \frac{2}{t+1}+\frac{1}{t}+\frac{r}{(t+1)^2}-\frac{r}{t(t+1)}$ for all $i \in [r]$. 
\end{proof}

\begin{center}
\section{Conjectures}
\end{center}
\begin{conjecture}
For all $r \in \mathbb{Z}^{+}$, in any optimal solution $a,b : \mathcal{P}([r]) \rightarrow [0,1]$ to the coloring problem, there do not exist $R_1,R_2 \subseteq [r]$ such that $|R_1 \cap R_2| \geq 2$ and $a(R_1) > 0, b(R_2) > 0$. 
\end{conjecture}

The reason why we believe this is true is because $|R_1 \cap R_2| \geq 2$ would imply that we are ``wasting" area. The region at the intersection of columns with colors $R_1$ and rows with colors $R_2$ can be colored with two different colors, but we really just need one of the colors to color the region. Note all of our constructions above satisfy the conjecture. \\

The next proposition shows the technique of row or column deletions in use. Note that deleting $0 < \epsilon < 1$ proportion of rows or columns to get a rectangle, and then rescaling the rectangle to become a square again gives us a new solution. If the old solution was optimal, and $a_i + b_i = k$ for some $i \in [r]$, then the new solution must also have some $j \in [r]$ where $a_j+b_j \geq k$ by the optimality of the old solution. We'll be more precise by what we mean by column and row deletions in the proposition below. 

\begin{proposition}
For all $r \in \mathbb{Z}^{+} \setminus \{1\}$, in any optimal solution $a,b : \mathcal{P}([r]) \rightarrow [0,1]$ to the coloring problem, there does not exist a color $i \in [r]$ such that $a_i+b_i > a_j+b_j$ for all $j \in [r] \setminus \{i\}$. 
\end{proposition}

\begin{proof}
 Assume for the sake of contradiction otherwise. Since $a_i+b_i > a_j+b_j \geq 0$ for all $j \in [r] \setminus \{i\}$ where $r \geq 2$, we have that there exists some $R \subseteq [r]$ such that $i \in R$ and $a(R) > 0$ or $b(R) > 0$. Without loss of generality, assume $a(R) > 0$. We redefine the function $a : \mathcal{P}([r]) \rightarrow [0,1]$ to get a better coloring. Let this function be $\alpha : \mathcal{P}([r]) \rightarrow [0,1]$. For some $0 < \epsilon < a(R)$, define $\alpha(R) = \frac{a(R)-\epsilon}{1-\epsilon}$ (column deletion and rescaling) and $\alpha(R') = \frac{a(R')}{1-\epsilon}$ (rescaling) for $R' \in \mathcal{P}([r]) \setminus \{R\}$. We can make $\epsilon$ small enough so that we still maintain $\alpha_i+b_i > \alpha_j + b_j$ for all $j \in [r] \setminus \{i\}$ under the coloring $\alpha, b : \mathcal{P}([r]) \rightarrow [0,1]$. But if $a_i < 1$, note that $\alpha_i+b_i < a_i + b_i$ because $\alpha_i = \frac{a_i-\epsilon}{1-\epsilon} < a_i$ , and thus $\alpha, b$ is a strictly better coloring than $a,b$. If $a_i = 1$, then we may perform row deletions instead, because we must have $0 < b_i < 1$, otherwise $a_i+b_i = 2$ but $g^*(r) < 2$ for $r > 1$. Then the rest of the argument follows by an analogous one given above, except for row deletions. \\
\end{proof}

We believe the following strengthening of Proposition 7 is true, and we suspect that the technique of row and column deletions may be helpful in proving it. \\

\begin{conjecture}
For all $r \in \mathbb{Z}^{+}$, in any optimal solution $a,b : \mathcal{P}([r]) \rightarrow [0,1]$ to the coloring problem, for all $i,j \in [r]$ we have that $a_i + b_i = a_j + b_j$. 
\end{conjecture}

We attempted to prove this conjecture but failed because in trying to prove the ``next step" of Proposition 7, namely that $a_1+b_1=a_2+b_2 > a_i+b_i$ for all $i \not \in \{1,2\}$ cannot hold, we run into more case work than in Proposition 7, and we believe further steps require even more casework which we currently don't know how to simplify. \\

We believe Theorems 7 and 8 provide an optimal solution as well. \\

\begin{conjecture}
Let $r \in \mathbb{Z}^{+}$. Suppose there exists $t \in \mathbb{Z}^{+}$ such that $t^2 \leq r \leq t(t+1)$. Then

$$g^*(r) =  \frac{2}{t} + \frac{1}{t+1} + \frac{r}{t(t+1)} - \frac{r}{t^2}.$$
\end{conjecture}

\begin{conjecture}
Let $r \in \mathbb{Z}^{+}$. Suppose there exists $t \in \mathbb{Z}^{+}$ such that $t(t+1) \leq r \leq (t+1)^2$. Then

$$g^*(r) = \frac{2}{t+1}+\frac{1}{t}+\frac{r}{(t+1)^2}-\frac{r}{t(t+1)}.$$
\end{conjecture}

\begin{center}
\section{Acknowledgements}
\end{center}

This research was conducted at the University of Minnesota Duluth REU and was supported by Jane Street Capital, NSF-DMS (Grant 1949884), Ray Sidney, and Eric Wepsic. We thank Joe Gallian, Colin Defant, Noah Kravitz, Maya Sankar, Mitchell Lee, and Ben Przybocki for feedback.

\bibliographystyle{plain}
\bibliography{refs}

@article{alon,
      title={The power of many colours}, 
      author={Noga Alon and Matija Bucić and Micha Christoph and Michael Krivelevich},
      year={2024},
      eprint={2308.15387},
      archivePrefix={arXiv},
      primaryClass={math.CO},
      url={https://arxiv.org/abs/2308.15387}, 
      journal={pre-print},
}

@article {liu,
    AUTHOR = {Liu, Henry and Morris, Robert and Prince, Noah},
     TITLE = {Highly connected monochromatic subgraphs of multicolored
              graphs},
   JOURNAL = {J. Graph Theory},
  FJOURNAL = {Journal of Graph Theory},
    VOLUME = {61},
      YEAR = {2009},
    NUMBER = {1},
     PAGES = {22--44},
      ISSN = {0364-9024,1097-0118},
   MRCLASS = {05C40 (05C15 05C35 05C55)},
  MRNUMBER = {2514097},
MRREVIEWER = {Daniel\ W.\ Cranston},
       DOI = {10.1002/jgt.20365},
       URL = {https://doi.org/10.1002/jgt.20365},
}

@misc{ration,
author = {Jiří Matoušek},
title = {Linear Programming},
url = {https://ti.inf.ethz.ch/ew/courses/APC21/Chapter_4.pdf}
}

\noindent
Department of Mathematics, Carnegie Mellon University, Pittsburgh, PA 15213, USA 
\end{document}